\documentclass[PL]{ejpecp_pl} 

\usepackage[T1]{fontenc}
\usepackage[utf8]{inputenc}

\usepackage{mathtools}
\DeclarePairedDelimiter\ceil{\lceil}{\rceil}
\DeclarePairedDelimiter\floor{\lfloor}{\rfloor}

\newcounter{propcounter}
\usepackage[usenames, dvipsnames]{color}
\definecolor{mygreen}{RGB}{0,150,0}

\usepackage{tikz}
\usepackage{tikz-3dplot}
\usetikzlibrary{calc,3d,arrows,shapes.geometric}
\usetikzlibrary{arrows.meta,calc,decorations.markings,math,arrows.meta}

\usepackage{pgfplots}

\newcommand{\comment}[1]{}
\usepackage{xspace}

\newcommand{\e}[0]{\mathbf{e}}
\newcommand{\PP}[0]{\mathbf{P}}

\newcommand{\Q}[0]{\mathbf{Q}}
\newcommand{\E}[0]{\mathbb{E}}
\newcommand{\Ex}[0]{\mathbf{E}}

\newcommand{\C}[0]{\mathbf{C}}
\newcommand{\B}[0]{\mathbf{B}}
\newcommand{\A}[0]{\mathbf{A}}
\newcommand{\D}[0]{\mathbf{D}}
\newcommand{\N}[0]{\mathbf{N}}

\newcommand{\ie}[0]{\textsl{i.e.,}\xspace}

\newcommand{\R}[0]{\mathbf{R}}
\newcommand{\I}[0]{\mathbf{I}}

\usepackage{bbm}

\newenvironment{equationprop}{%
\addtocounter{equation}{-1}
\refstepcounter{propcounter}

\begin{equation}}
{\end{equation}}

\SHORTTITLE{Multidimensional gambler models} 

\TITLE{Absorption time and absorption probabilities for a family of multidimensional gambler models\thanks{Work supported 
by NCN Research Grant DEC-2013/10/E/ST1/00359.}} 

\AUTHORS{%
 Pawe\l~Lorek\footnote{Mathematical Institute, Wrocław University, Wrocław, Poland. \EMAIL{Pawel.Lorek@math.uni.wroc.pl, Piotr.Markowski@math.uni.wroc.pl}}
 \and 
 Piotr Markowski\footnotemark[2]}
 

\AMSSUBJ{60J10} 
\AMSSUBJSECONDARY{ 60G40; 60J80} 

\SUBMITTED{July 18, 2018} 
\ACCEPTED{...} 

\VOLUME{0}
\YEAR{2018}
\PAPERNUM{0}
\DOI{.........}

\ABSTRACT{For a family of multidimensional gambler models we provide formulas for the winning probabilities (in terms of 
parameters of the system) and for the  distribution of game duration (in terms of eigenvalues of underlying one-dimensional 
games). These formulas were known for one-dimensional case - initially proofs were purely analytical, later probabilistic construction has been given. Concerning the game duration, in many cases our approach yields sample-path constructions.
We heavily exploit intertwining between (not necessary) stochastic matrices (for game duration results), a notion of 
Siegmund duality (for winning/ruin probabilities), and a notion of Kronecker products.}

\begin{document}
 

 
\section{Introduction}

In the one-dimensional gambler's ruin problem two players start a game with a total amount of, say, $N$ dollars and initial values   $k$ and $N-k$.
At each step they flip the   coin (not necessary unbiased) to decide who wins a dollar. 
The game is over when one of them goes bankrupt.
There are some fundamental questions related to this process.
\begin{itemize}
 \item[Q1] Starting with $i$ dollars, what is the probability of winning?
 \item[Q2] Starting with $i$ dollars, what is the distribution (or structure) of game duration  (\ie  absorption time)? Or,
 what  is the distribution (or structure) of game duration  conditioned on winning/losing?
\end{itemize}
In this paper we will answer above questions for a wide class of multidimensional generalizations of gambler's ruin problem. 
The proofs will be probabilistic in most cases, utilizing either Siegmund duality or intertwining between chains.
\par

\paragraph{Generalized multidimensional gambler models}
In \cite{2015Lorek_gambler} we considered the following generalization.  There is one player (referred as ``we'') playing with $d\geq 1$ other players.
Our initial assets are $(i_1,\ldots,i_d)$ and assets of consecutive players are $(N_1-i_1,\ldots,N_d-i_d)$
($N_j\geq 1$ is a total amount of assets with player $j$).
Then, with probability $p_j(i_j)$ we win one dollar with player $j$ and with probability $q_j(i_j)$ we lose it.
With the remaining probability $1-\sum_{k=1}^d(p_k(i_j)+q_j(i_k))$ we do nothing (i.e., ties are also possible).
Once we win completely with player $j$ (i.e., $i_j=N_j$) we do not play with him/her anymore.
We lose the whole game if we lose with at least one player, i.e., when $i_j=0$ for some $j=1,\ldots,d$.
The game can be described more formally as a Markov chain $Z$ with two absorbing states. 
 The state space is $\E=\{(i_1,\ldots, i_d): 1\leq i_j\leq N_j, 1\leq j\leq d\}\cup\{-\infty\} $ (where $-\infty$ means we \textsl{lose}).
 For convenience denote $p_j(N_j)=q_j(N_j)=0, j=1,\ldots,d$. Assume  that for all $i_j\in\{1,\ldots,N_j\}, j\in\{1,\ldots,d\}$ we have  
 $p_j(i_j)>0, q_j(i_j)>0 $ and  
 $  \sum_{k=1}^d (p_k(i_k)+q_k(i_k)) \leq 1. $
  With some abuse of notation, we will sometimes write $ (i'_1,\ldots,i'_d)=-\infty$. 
  The transitions of the described chain are the following: \medskip \par 
  \noindent
  $\displaystyle \PP_{Z}((i_1,\ldots,i_j),(i'_1,\ldots,i'_j)) =$
\begin{equation}\label{eq:PZp} 
 \left\{ 
 \begin{array}{llllllll}
  p_j(i_j) & \textrm{if} & i_j'=i_j+1, i_k'=i_k, k\neq j, \\[3pt]
  q_j(i_j) & \textrm{if} & i_j'=i_j-1, i_k'=i_k, k\neq j, \\[3pt]
  \sum_{j:i_j=1}q_j(1) & \textrm{if} & (i'_1,\ldots,i'_j)=-\infty, \\[3pt]
  1-\sum_{k=1}^d (p_k(i_k)+q_k(i_k))& \textrm{if} & i_j'=i_j, 1\leq j\leq d,\\[3pt]
  1&\textrm{if} & (i_1,\ldots,i_j)=(i'_1,\ldots,i'_j)=-\infty.    
 \end{array}
 \right.
\end{equation}
The chain  has two absorbing states: $(N_1,\ldots,N_d)$ (we \textsl{win}) and $-\infty$ (we \textsl{lose}).
Let 
\begin{equation}\label{eq:rho} 
\rho((i_1,\ldots,i_d))=P(\tau_{(N_1,\ldots,N_d)} < \tau_{-\infty} | Z_0=(i_1,\ldots,i_d)), 
\end{equation}
where $\tau_{(i_1',\ldots,i_d')}:=\inf\{n\geq 0: Z_n={(i_1',\ldots,i_d')}\}$. Roughly speaking, $\rho((i_1,\ldots,i_d))$ is the probability of winning starting at $(i_1,\ldots,i_d)$.
In \cite{2015Lorek_gambler} we derived the  formula for this probability, namely
\begin{equation}\label{eq:gabmler_rho}
\rho((i_1,\ldots,i_d))=
{
\displaystyle \prod_{j=1}^d \left(  \sum_{n_j=1}^{i_j}  \prod_{r=1}^{n_j} \left({q_j(r)\over p_j(r)}\right) \right)
\over
\displaystyle \prod_{j=1}^d \left(  \sum_{n_j=1}^{N_j} \prod_{r=1}^{n_j} \left({q_j(r)\over p_j(r)}\right)
\right) }.
\end{equation}
In this paper we consider  a much  wider class of $d$-dimensional games - the chain given in (\ref{eq:PZp}) 
is just a special case. For example, within the class we can win/lose in one step 
with many players. The multidimensional chain is constructed from 
a variety of one-dimensional chains using    Kronecker products. For this class:
\begin{itemize}
 \item We give expressions for the winning probabilities and prove that it  is a   
 product of the winning probabilities corresponding 
 to one-dimensional games. 
 In particular, for a subclass of the multidimensional chains, constructed from one-dimensional birth and death chains,
 the winning probabilities are   given in  (\ref{eq:gabmler_rho}). 
 The main tool for showing winning probabilities is the Siegmund duality defined for partially ordered 
 state spaces, exploiting the results from \cite{Lorek2016_Siegmund_duality}.
 
 \item We give formulas for the distributions of game duration. In some cases a probability generating function is given,
 in other cases we show that the absorption time is equal, in distribution, to the absorption time of another chain,
 which is, in a sense, a multidimensional pure-birth chain. In many cases, the probabilistic proof is given.
 To show the absorption distribution, we exploit the spectral polynomials given in \cite{Fill2009}, and 
 their variations considered in \cite{Gong2012}, \cite{Mao2016}.
  
\end{itemize}

\begin{remark}
 \rm In \cite{2015Lorek_gambler} we considered the chain -- given in (\ref{eq:PZp}) -- which is 
 constructed from $d$ one-dimensional birth and death chains in a very specific way. The method 
 from this article is much more general, we can construct many different multidimensional chains 
 from given  $d$ one-dimensional birth and death chains. It is worth mentioning, that even 
 for the case (\ref{eq:PZp}), the proof is quite different (from the one in \cite{2015Lorek_gambler}). 
\end{remark}

Several variations (including multidimensional ones) of gambler's ruin problem have been considered. Researchers usually study absorption probabilities, absorption time, or both.
In \cite{Kmet2002}  authors consider two-dimensional model (they consider two currencies) and study expected game duration.
In \cite{Ross2009} some multidimensional game is considered:   at each step two players are randomly chosen,
these players play a regular game, all till one of the players have all the coins. Author derives the probability that specific player wins, the expected number of turns in total and 
between two given players.
In \cite{Rocha2004}  the following multidimensional game is considered:  there are $n$ players, at each step there is one winner 
which collects $n-1$ coins from other players, whereas all others lose $1$ coin. Asymptotic probability for individual ruin and dependence of
ruin time are studied. 
In \cite{Tzioufas2016} the multidimensional case is considered, in which with equal probability a unit displacement in any direction is possible. 
Moments of leaving some a ball are considered.
In \cite{Champagnat} authors present a new probabilistic analysis of distributed algorithm re-considering a variation of banker algorithm.
Mathematically, it is random walk on a rectangle with specified absorbing states. The results are generalized to the case with many players and  resources. 

\par 
The absorption probability of given chain may be related to the stationary distribution of some ergodic chain.
This relation is given using \textsl{Siegmund duality}, the notion introduced in  \cite{Siegmund1976}. 
This is also the tool we use for showing absorption probabilities.
Already in \cite{Lindley1952a} similar duality between some random walks on integers was shown.  It was also studied in financial mathematics,
where the probability that a dual risk process starting at some level is ruined, is equal to the probability that the stationary 
queue length exceeds this level (see \cite{Asmussen2010}, \cite{Asmussen2009}).
In all these cases the Siegmund duality was defined for linear ordering of the state space. The existence of a Siegmund dual 
for linearly ordered state space requires stochastic monotonicity of a chain. Recently, in \cite{Lorek2016_Siegmund_duality} we provided 
\textsl{if and only if} conditions for existence of  Siegmund dual for partially ordered state spaces (roughly speaking, 
the M\"obius monotonicity is required). In this paper, we exploit this duality  defined for a coordinate-wise partial ordering.

\paragraph{Absorption time}
Consider   one-dimensional game corresponding to the gambler's ruin problem. 
Let $N$    be  a total amount of money. Being at state $i\in\{2,\ldots,N-1\}$ we can either win one dollar with probability $p(i)>0$ or lose it 
with probability $q(i)>0$, with the remaining probability nothing happens. Assuming $p(1)>0$ and $p(N)=q(N)=p(0)=q(0)=0$
the transitions are the following:
\begin{equation}\label{eq:PZj} 
 \displaystyle \PP_{Y}(i,i')= \left\{ 
 \begin{array}{llllllll}
  p(i) & \textrm{if} & i'=i+1, \\[3pt]
  q(i) & \textrm{if} & i'=i-1,  \\[3pt]
  1-(p(i)+q(i))& \textrm{if} & i'=i.\\[3pt]
 \end{array}
 \right.
\end{equation}
States 0 and $N$ are absorbing. 
Consider two cases:
\paragraph{Case: $q(1)=0$} Roughly speaking,
if started in $i\geq 1$ the chain never reaches $0$ and this is actually a birth and death chain on $\{1,\ldots,N\}$ with $N$ being the only absorbing state. 
Define   $T_{a,b}=\inf\{n\geq 0:  Y_n=b \ | Y_0=a \ \}$.
A well known theorem attributed to Keilson \cite{Keilson79} states, that the probability generating function $pgf$ of $T_{1,N}$ is the following:
\begin{equation}\label{eq:gT1}
pgf_{T_{1,N}}(u):=\Ex u^{T_{1,N}}=\prod_{k=1}^{N-1} \left[{(1-\lambda_k)u\over 1-\lambda_k u}\right], 
\end{equation}
where $-1\leq\lambda_k<1, k=1,\ldots, N-1$ are $N-1$ non-unit eigenvalues of $\PP_{Y}$.
The proof was purely analytical. Note that (\ref{eq:gT1}) corresponds to the sum of $N_j$ geometric random variables, 
provided that all eigenvalues are positive (which, in this case, is equivalent to the stochastic monotonicity of the chain).
For this case,  Fill \cite{Fill2009} (in 2009) gave a probabilistic proof of (\ref{eq:gT1}) using strong stationary duality
and intertwinings between chains.
Note that in this  case  (\ref{eq:gT1}) can be rephrased as:
\begin{theorem}\label{thm:FillMatrix}

Let $X^*$ be an absorbing chain on $\E=\{1,\ldots,N\}$ starting at $1$ with transition matrix  $\PP_{X^*}$ given in (\ref{eq:PZj}) having positive eigenvalues 
$\lambda_k>0,k=1,\ldots,N$.  Then $T^*_{1,N}$ has the same distribution as $\hat{T}_{1,N}$, the absorption time of $\hat{X}$ on $\hat{\E}=\E$ starting at $1$ with 
transition matrix
\begin{equation*}\label{eq:PXd} 
 \displaystyle \PP_{\hat{X}}(i,i')= \left\{ 
 \begin{array}{llllllll}
  1-\lambda_i & \mathrm{if}\  i'=i+1, \\[3pt]
  \lambda_i & \mathrm{if}\  i'=i,  \\[3pt]
  0& \mathrm{otherwise}.\\[3pt]
 \end{array}
 \right.
\end{equation*}
\end{theorem}
The chain $Y$ on $\{1,\ldots,N\}$ is called  pure birth if $\PP_Y(i,j)=0$ for $j\leq i$.
Similarly, a multidimensional chain $Y$ on $\E=\{(i_1,\ldots,i_d): 1\leq i_j\leq N_j, 1\leq j\leq d\}$
($N_j$ can be $\infty$) is said to be \textbf{pure birth} if the probability of decreasing 
any set of coordinates at one step is 0.

\medskip\par 
Simply noting that for any $1<s<N$ we have $T_{1,N}=T_{1,s}+T_{s,N}$ and that $T_{1,s}$ and $T_{s,N}$ are independent
(derived in 2012, see  Cor. 2.1 \cite{Gong2012} for continuous time version) we have 
\begin{equation}\label{eq:gTi}
pgf_{T_{s,N}}(u):=\Ex u^{T_{s,N}}={\displaystyle \prod_{k=1}^{N-1} \left[{(1-\lambda_k)u\over 1-\lambda_k u}\right]\over \displaystyle \prod_{k=1}^{s-1} \left[{(1-\lambda^{\floor{s}}_k)u\over 1-\lambda^{\floor{s}}_k u}\right]}, 
\end{equation}
where $\lambda_k^{\floor{i}}$ are the eigenvalues of substochastic $(s-1)\times (s-1)$ matrix
$$
 \displaystyle \PP_{Y}^{\floor{s}}(i,i')= \left\{ 
 \begin{array}{llllllll}
  p(i) & \textrm{if} & i'=i+1, 1\leq i\leq s-2,  \\[3pt]
  q(i) & \textrm{if} & i'=i-1, 2\leq i\leq s-1 \\[3pt]
  1-(p(i)+q(i))& \textrm{if} & i'=i, 1\leq i\leq s-1.\\[3pt]
 \end{array}
 \right.
$$

\paragraph{Case: $q(1)>0$}   In this case,  authors in \cite{Gong2012} (different proof  is given in \cite{Mao2016}) derived 
formulas for $pgf$ of ${T_{s,N}}$ and ${T_{s,0}}$ (more precisely, they derived formulas for continuous time versions), which,
in discrete case, are given by
\begin{align} 
&pgf_{T_{s,N}}(u) =\Ex u^{T_{s,N}}=& \rho(s){\displaystyle \prod_{k=1}^{N-1} \left[{(1-\lambda_k)u\over 1-\lambda_k u}\right]\over \displaystyle \prod_{k=1}^{s-1} \left[{(1-\lambda^{\floor{s}}_k)u\over 1-\lambda^{\floor{s}}_k u}\right]}, & 
 \label{eq:pgf_TsN_j}\\[10pt]
  \qquad \qquad \qquad \qquad &pgf_{T_{s,0}}(u)=\Ex u^{T_{s,0}}=&(1-\rho(s)){\displaystyle \prod_{k=1}^{N-1} \left[{(1-\lambda_k)u\over 1-\lambda_k u}\right]\over \displaystyle \prod_{k=1}^{N-s-1} \left[{(1-\lambda^{\ceil{s}}_k)u\over 1-\lambda^{\ceil{s}}_k u}\right]}, & \qquad \qquad \qquad \qquad 
\nonumber
\end{align}
where $\rho(s)$ is the probability of winning 
(\ie (\ref{eq:rho}) with $d=1, i_1=s$) and $\lambda^{\ceil{s}}_k$ are the eigenvalues of substochastic matrix (of the size $N-s-1$)
$$
 \displaystyle \PP_{Y}^{\ceil{s}}(i,i')= \left\{ 
 \begin{array}{llllllll}
  p(i) & \textrm{if} & i'=i+1, s+1\leq i\leq N-2,  \\[3pt]
  q(i) & \textrm{if} & i'=i-1, s+2\leq i\leq N-1 \\[3pt]
  1-(p(i)+q(i))& \textrm{if} & i'=i, s+1\leq i\leq N-1.\\[3pt]
 \end{array}
 \right.
$$
In this paper we aim at presenting results similar to Theorem \ref{thm:FillMatrix} and to (\ref{eq:pgf_TsN_j}) for 
a wide class of multidimensional extensions of gambler's ruin problem.
%
%

\section{Kronecker product and main results}\label{sec:main_results}

To state our main results we need to recall a notion of Kronecker product.
Let $\A$ be a matrix of size $n\times m$. Then, for any matrix $\B$ the Kronecker product of the matrices is defined as follows:
$$\A\otimes \B=
\begin{bmatrix}
    a_{11}\B       & a_{12}\B  & \dots & a_{1m}\B \\
    a_{21}\B       & a_{22}\B  & \dots & a_{2m}\B \\
    \hdotsfor{4} \\
    a_{n1} \B      & a_{n2}\B  & \dots & a_{nm}\B
\end{bmatrix}$$
For square matrices $\A$ and $\B$ it is also convenient to define the Kronecker sum as:
$$ \A\oplus \B= \A\otimes \I_\B + \I_\A \otimes \B,$$
where $\I_\A$ ($\I_\B$) is the identity matrix of the same size as $\A$ ($\B$).\par
Both, product and sum, are extended as:
$$\bigotimes_{i=1}^n \A_i = (\ldots((\A_1\otimes \A_2)\otimes \A_3)\ldots)\otimes \A_n = \A_1\otimes \A_2 \otimes \ldots \otimes \A_n $$
and
$$\bigoplus_{i=1}^n \A_i = (\ldots((\A_1\oplus \A_2)\oplus \A_3)\ldots)\oplus \A_n = \A_1\oplus \A_2 \oplus \ldots \oplus \A_n.$$

\par 
\paragraph{Notation}
For a convenience, for given substochastic matrix  $\PP'_{Y}$ on $\E'=\{\e_1,\ldots,\e_M\}$ by 
$\PP_{Y}=\mathcal{F}_{\e_0}(\PP'_{Y})$ we denote a
stochastic matrix on $\E=\{\e_0\}\cup \E'$ constructed from $\PP'_{Y}$ in the following way:
\begin{equation*}\label{eq:StochMat}
\PP_{Y}(i,j)= \left\{ 
 \begin{array}{llllllll}
  \PP'_{Y}(\e_i,\e_j) & \textrm{if} & \e_i,\e_j\in \E, \\[3pt]
  1-\sum_{\e_k\in\E'}\PP'_{Y}(\e_i,\e_k) & \textrm{if} & \e_i\in \E', \e_j=\e_0, \\[3pt]
  1 & \textrm{if} & \e_i=\e_j=\e_0. \\[3pt]
  0 & \textrm{if} & \e_i=\e_0, \e_j\in \E. \\[3pt]
 \end{array}
 \right.
\end{equation*}
Similarly, for a stochastic matrix $\PP_{Y}$ on $\E=\{\e_0\}\cup\E'$ let $\PP'_{Y}=\mathcal{F}^{-1}_{\e_0}(\PP_{Y})$ 
be a substochastic matrix on 
$\E'$ resulting from $\PP_{Y}$ by removing row and  column   corresponding to the state $\e_0$.
\medskip\par
For a Markov chain $Y$ on $\E=\{\e_1,\ldots,\e_M\}$ we say that $A\subseteq\E$ is a communication class
if for all $\e,\e'\in A$ we have  $\PP_Y^n(\e,\e')>0$ for some $n\geq 0$.
\medskip\par
For given chain $Y$ define $T_{\nu,\e'}:=\inf\{n\geq 0: Y_n=\e' | Y_0 \sim \nu\}$.
Slightly abusing the  notation, by $T_{\e,\e'}$ we mean $T_{\nu,\e'}$ with $\nu=\delta_\e$.

For $\E=\{\e_1,\ldots,\e_M\}$ and for $f:\E\to\mathbb{R}$, we define a row vector $\boldsymbol{f}=(f(\e_1),\ldots,f(\e_M))$.

\par For $N_j>0, j=1,\ldots,d$ define $\N=(N_1,\ldots,N_d)$.

\par\medskip
\noindent
Now we are ready to state our main results.
\subsection{Absorption probabilities}

\begin{theorem}\label{thm:main_rho}
Fix integers $d\geq 1, m\geq 1$. For $k=1,\ldots,m$ let $A_k\subseteq\{1,\ldots,d\}$.
Assume 
\begin{itemize}
 \item $\forall(1\leq k\leq m)$ $\PP_{Z^{(k)}_j}=\mathcal{F}_0(\PP'_{Z^{(k)}_j})$ is a stochastic matrix 
 corresponding to a Markov chain 
 $Z_j^{(k)}$ on $\E_j=\{0,1,\ldots,N_j\}$ such that  for $i\in\E_j$ we have 
 \begin{equation}\label{eq:rho_k}
\rho^{(k)}_j(i) = P(\tau_{N_j}<\tau_0 | Z_j^{(k)}(0)=i)=\rho_j(i).  
 \end{equation}

  In other words,
  $Z_j^{(k)}$  are $m$ ($k=1,\ldots,m$) chains having the same winning probability at every state $i$.

 \item Let 
 $$\R'_{Z_j^{(k)}} = 
 \left\{ 
 \begin{array}{lll}
  \PP'_{Z_j^{(k)}} & \mathrm{if\ } j\in A_k,\\[10pt]
  \I_j & \mathrm{if\ } j\notin A_k,\\  
 \end{array}\right.
$$
where $\I_j$ is the identity matrix of size $N_j\times N_j$. 
\item Let $\B_i, i=1,\ldots,m$ be either 
\begin{itemize}
\item any real numbers (\ie $\B_k\in\mathbb{R}$) such that $\sum_{k=1}^m \B_k=1$, \textbf{or}
\item square matrices of size $\prod_{j=1}^d N_j \times \prod_{j=1}^d N_j $ such that $\sum_{k=1}^m \B_k=\I$ (identity matrix of 
the appropriate size)
\end{itemize}
\item The matrix $\PP_Z=\mathcal{F}_{-\infty}(\PP_{Z'})$ with 
\begin{equation*}\label{eq:PZ_general}
\PP'_{Z} =\sum_{k=1}^m \B_k\left(\bigotimes_{j\leq d} \R'_{Z_j^{(k)}}\right) 
\end{equation*}
is stochastic on   $\E=\{-\infty\}\cup\bigotimes_{j\leq d} \E_j', $   
set $\E\setminus\{\{\N\}\cup\{-\infty\}\}$ is a communication class.

\end{itemize}
Then, the winning probability (\ie absorption at $\N$) of the  Markov chain 
$\mathbf{Z}$ on $\E=\{-\infty\}\cup\{1,\ldots,N_1\}\times\ldots\times \{1,\ldots,N_d\}$ with 
transition matrix  
$\PP_{Z}=\mathcal{F}_{-\infty}(\PP'_{Z})$ is given by
\begin{equation}\label{eq:rho_product}
\rho(i_1,\ldots,i_d)=\prod_{j=1}^d \rho_j(i_j).
\end{equation}
%
%
\end{theorem}
\noindent
The proof is postponed to Section \ref{sec:proof_rho}.
\par 
Note that $\PP_{Z_j^{(k)}}$ in Theorem \ref{thm:main_rho} are  general. If we only know 
the winning probabilities of $\PP_{Z_j^{(k)}}$ (they cannot depend on $k$), then we know 
the winning probabilities of $Z$. Taking $\PP_{Z_j^{(k)}}$ corresponding to gambler's ruin game given in (\ref{eq:PZj}) we have:
\begin{corollary}\label{cor:rho_bd_case}
 Let $\PP_{Z_j^{(k)}}$ for $j=1,\ldots,d$ be the birth and death chain given in  (\ref{eq:PZj}).
Then, the winning probability of $\PP_{Z}=\mathcal{F}_{-\infty}(\PP'_{Z})$
is given by (\ref{eq:gabmler_rho}).
 \end{corollary}
 
 \begin{proof}
  For $\PP_{Z_j^{(k)}}$ the winning probability is known (shown in (\ref{eq:1d_rho})), it is 
  \begin{equation}\label{eq:rho_bd}
   \rho_j(i_j)={\displaystyle \sum_{n_j=1}^{i_j} \prod_{r=1}^{n_j-1} \left( {q_j(r)\over p_j(r)}\right)
  \over \displaystyle \sum_{n_j=1}^{N_j} \prod_{r=1}^{n_j-1} \left( {q_j(r)\over p_j(r)}\right)}.
  \end{equation}

  Assertion of Theorem \ref{thm:main_rho}  completes the proof.
 \end{proof}

 The chain $\mathbf{Z}$ can be interpreted as $d$-dimensional game, with state $(N_1,\ldots,N_d)$ corresponding to winning and state $-\infty$ corresponding to losing.

\subsection{Absorption time}
\noindent
We have the following extension of Theorem \ref{thm:FillMatrix}   to the multidimensional case:

\begin{theorem}\label{thm:main_abs_times}
Fix integers $d\geq 1, m\geq 1$. For $k=1,\ldots,m$ let $A_k\subseteq\{1,\ldots,d\}$.
 Let $b_i \in\mathbb{R}, i=1,\ldots,m$ such that   $\sum_{k=1}^m b_i=1$. 
 Let, for $1\leq j\leq d$, $\PP_{X^*_j}$ be a stochastic matrix corresponding to a birth and 
 death chain $X_j^*$ on $\E_j=\{0,\ldots,N_j\}$ with transitions given in (\ref{eq:PZj})
 with birth rates $p_j(i)$ and death rates $q_j(i)$. 
 Let, for $1\leq j\leq d$,  $\PP_{X^*_j}'=\mathcal{F}^{-1}_0(\PP_{X^*_j})$ be a substochastic matrix on $\E_j'=\{1,\ldots,N_j\}$ and
 $$\R_{X_j^{*(k)}}' = 
 \left\{ 
 \begin{array}{lll}
  \PP_{X^*_j}' & \mathrm{if\ } j\in A_k,\\[10pt]
  \I_j & \mathrm{if\ } j\notin A_k,\\  
 \end{array}\right.
$$
where $\I_j$ is the identity matrix of size $N_j\times N_j$. \textsl{I.e.,} $\R_{X^{*(k)}_j}'$ is either matrix $\PP_{X^*_j}'$ 
or an identity matrix.
Let $\lambda_1^{(j)}\leq \ldots \leq \lambda_{N_j-1}^{(j)} < \lambda_{N_j}^{(j)}=1$ be the eigenvalues of $\PP'_{X^*_j}$.

\par

Assume  
\begin{itemize}
 \item[\textbf{A1}]  
 The chains $\PP_{X^*_j}, j=1,\ldots,d$ are stochastically monotone.
 
 
\item[\textbf{A2}] The  matrix $\PP_{X^*}=\mathcal{F}_{-\infty}(\PP_{X^*}')$ with
\begin{equation}\label{eq:PZ_general2}
\PP_{X^*}' =\sum_{k=1}^m b_k\left(\bigotimes_{j\leq d} \R_{X_j^{*(k)}}'\right) 
\end{equation}
is stochastic matrix on $\E=\{-\infty\}\cup\bigotimes_{j\leq d} \E_j', $   
set $\E\setminus\{\{\N\}\cup\{-\infty\}\}$ is a communication class, where $\N=(N_1,\ldots,N_d)$.
\item[\textbf{A3}] The  matrix $\PP_{\hat{X}}$, given below in (\ref{eq:abs_time_res_prod}), is non-negative. 

\end{itemize}
Let $X^*$ be a chain with the above transition matrix $\PP_{X^*}$.
  Assume its initial distribution is $\nu^*$.
 The state $\N $ 
 is absorbing state,
denote its absorption time by $T^*_{\nu^*,\N}$. 
\par \noindent
Then the time to absorption $T^*_{\nu^*,\N}$ has the following pgf
$$ pgf_{T_{\nu^*,\N}^*}(s)=
\sum_{\hat{\e}\in {\E}} 
\hat{\nu}(\hat{\e}) pgf_{\hat{T}_{\hat{\e},\N}}(s) \left(\prod_{j=1}^d \rho_j(1)\right),$$
where $\rho_j(1)$ is the winning probability of $X_j^*$ starting at $1$,
$$\hat{\nu}=\nu^*\bigotimes_{j\leq d}\Lambda^{-1}_j,$$
$\Lambda_j$ are given in (\ref{eq:link_spectral}) calculated for $\PP_{X^*_j}'$
and 
$\hat{T}_{\hat{\e},\N}$ is time to absorption in the chain $\hat{X} \sim (\delta_{\hat{\e}},\PP_{\hat{X}})$ with:
\bigskip\par\noindent
$\PP_{\hat{X}}((i_1,\ldots,i_d),(i'_1,\ldots, i'_d)) =$
\begin{equation}\label{eq:abs_time_res_prod}
\left\{
\begin{array}{lll}
\displaystyle   \prod_{j\in B} \left(1-\lambda_{i_j}^{(j)}\right)  \sum_{k: B\subseteq A_k}\left(b_k  \prod_{j\in A_k \setminus B} \lambda_{i_j}^{(j)}\right) 
& \textrm{if } \ i'_j=i_j+1, \\
& \quad  j\in B\subseteq\{1,\ldots,d\}, B\neq \emptyset \\[16pt]
\displaystyle   \sum_{k=1}^m b_k  \prod_{j\in A_k} \lambda_{i_j}^{(j)} 
& \textrm{if } \ i'_j=i_j \textrm{ for } j=1,\ldots,d.\\[18pt]
0 & \textrm{otherwise}.
\end{array}
\right. 
\end{equation}
We also have
\begin{equation}\label{eq:vg_b}
\forall(\e\in\E)  {\nu}^*(\e)\neq 0 \ \Rightarrow \ \exists(\e'\succeq \e) \hat{\nu}(\e')>0
\end{equation}
Moreover, the eigenvalues of $\PP_{X^*}$  and $\PP_{\hat{X}}$  are the diagonal entries of $\PP_{\hat{X}}$.

%
\end{theorem}
\noindent
Note that $\hat X$ is a pure-birth chain. Moreover, at one step it can change increase coordinates by +1 
on a set $B$ such that $B\subseteq A_k, $ for $k=1,\ldots,m.$\par 

\begin{remark}\rm
 In case $b_i\geq 0, i=1,\ldots,m$ (\ie $(b_1,\ldots,b_m)$ is a distribution on $\{1,\ldots,m\}$) the matrix $\PP_{X^*}$ 
 in assumption \textbf{A2} is  stochastic (thus  \textbf{A2}
 is only about $\E\setminus\{\{\N\}\cup\{-\infty\}\}$ being a communication class) 
 and so is the matrix $\PP_{\hat{X}}$ given  in (\ref{eq:abs_time_res_prod})
 (\ie \textbf{A3} is fulfilled).

\end{remark}

Considering initial distribution having whole mass at $(1,\ldots,1)$ and/or 
all $q_j(1)=0, j=1,\ldots,d$ we have special cases, which we will formulate as a corollary.

\begin{corollary}\label{cor:abc}
 Consider setup from Theorem \ref{thm:main_abs_times}. 
 
 \begin{itemize}
  \item[a)] Moreover, assume that $q_j(1)=0$ for all $j=1,\ldots,d$.
 I.e., each $X^*_j$ has actually only one absorbing state (state $0$ is not accessible).
 Then, $\N$ is the only absorbing state of $X^*$,  
 $\sum_{\e\in\E}\hat{\nu}(\e)$=1, $\rho_j(1)=1, j=1,\ldots,d$ and  we have  
 $$ pgf_{T_{\nu^*,\N}^*}(s)=
\sum_{\hat{\e}\in {\E}} \hat{\nu}(\hat{\e}) pgf_{\hat{T}_{\hat{\e},\N}}(s).$$

\item[b)] 
 Moreover, assume that both $q_j(1)=0$ for all $j=1,\ldots,d$
  and 
 $\nu^*((1,\ldots,1))=1$. Then 
 $T^*_{(1,\ldots,1),\N}\stackrel{d}{=}\hat{T}_{(1,\ldots,1),\N}$,
 where $\stackrel{d}{=}$ denotes equality in distribution.
 
 \item[c)]
  Moreover, assume that $\nu^*((1,\ldots,1))=1$. Then,
 assertions of Theorem \ref{thm:main_abs_times} hold with $\hat{\nu}((1,\ldots,1))=1$ and we have 
 $$
T^*_{(1,\ldots,1),\N}=\left\{ 
 \begin{array}{llll}
\hat{T}_{(1,\ldots,1),\N} & \mathrm{with\ probability\ }  \prod_{j=1}^d \rho_j(1), \\[8pt]
+\infty   & \mathrm{with\ probability\ }  1-\prod_{j=1}^d \rho_j(1). \\ 
\end{array}
\right.  $$  
 \end{itemize}

\end{corollary}

\paragraph{Sample-path construction}
It turns out that when $\hat{\nu}$ resulting from $\hat{\nu}=\nu^*\Lambda^{-1}$ is a distribution (which is always the case in, e.g.,
 Corollary \ref{cor:abc} \ \textsl{b)} and \textsl{c)}, we can have a sample-path construction. I.e., for $X^*$ we can construct, sample path by sample path,  
 a chain $\hat{X}$, so that $T^*_{\nu^*,\N}$ has the distribution expressed in terms of $\hat{T}_{\hat\nu,\N}$ as stated in Theorem \ref{thm:main_abs_times}.
 The construction is analogous to the construction given in \cite{Diaconis1990a} (paragraph 2.4) - note however that the construction therein was 
 between ergodic chain and its strong stationary dual chain (i.e., the chain with one absorbing state) and the link $\Lambda$ was a stochastic matrix 
 (it can be substochastic in our case). Having observed
 Having observed $X^*_0=\e^*_0$ (chosen from the distribution $\nu^*$) we set 
 $$\hat{X}_0=\hat{\e}_0 \textrm{ with probability } \frac{\hat{\nu}(\hat{\e}_0)\Lambda(\hat{\e}_0,\e^*_0)}{\nu^*(\e^*_0)}.$$
 Then, after choosing $X^*_1=\e^*_1,\ldots,X^*_{n-1}=\e^*_{n-1}$ and $\hat X_1=\hat \e_1,\ldots,\hat X_{n}= \hat \e_{n}$ we set 
 $$\hat{X}_n=\hat{\e}_n \textrm{ with probability } \frac{\PP_{\hat{X}}(\hat{\e}_{n-1},\hat{\e}_{n})\Lambda(\hat{\e}_n,\e^*_n)}{(\PP_{X^*}\Lambda)(\hat{\e}_{n-1},\e^*_n)}.$$
 This way we have constructed the chain $\hat{X}$ so that $\Lambda \PP_{X^*}=\PP_{\hat{X}}\Lambda$ and $\nu^*=\hat{\nu}\Lambda$ with the property that 
   $\hat{X}_n=\hat \e_M$ if and only if $X^*_n=\e^*_M$.

\medskip
Theorem \ref{thm:main_abs_times} is actually neither   an extension of (\ref{eq:gTi}) nor (\ref{eq:pgf_TsN_j}) to the multidimensional case, since 
for one-dimensional case the formula for pgf of $T^*_{s,N}$ has a  different form, as the example 
given in Section \ref{sec:example_b} shows.



\section{Tools: dualities in Markov chains}
Siegmund duality and intertwinings between chains  are the key ingredients of our main theorems' proofs.
\subsection{Siegmund duality}\label{sec:siegmund}
Let $X$ be an ergodic discrete-time Markov chain with transition matrix $\PP_X$ and finite state space $\E=\{\e_1,\ldots,\e_M\}$ partially ordered by $\preceq$.
Denote its stationary distribution by $\pi$. We assume that there exists a unique minimal state,
say $\e_1$, and a unique maximal state, say $\e_M$. For $A\subseteq \E$, define $\PP_X(\e,A):=\sum_{\e'\in A}\PP_X(\e,\e')$ 
and similarly $\pi(A):=\sum_{\e\in A}\pi(\e)$. Define also $\{\e\}^\uparrow:=\{\e'\in\E: \e\preceq \e'\}$,
$\{\e\}^\downarrow:=\{\e'\in\E: \e'\preceq \e \}$ and $\delta(\e,\e')=\mathbbm{1}\{\e=\e'\}$.  
We say that a Markov chain $Z$ with 
transition matrix $\PP_Z$ is the 
\textbf{Siegmund dual} of $X$ if 
\begin{equation}\label{eq:Siegmund_duality}
 \forall(\e_i,\e_j\in \E)\ \forall(n\geq 0) \quad \PP^n_X(\e_i, \{\e_j\}^\downarrow) = \PP^n_Z(\e_j,\{\e_i\}^\uparrow).
\end{equation}
In all non-degenerated applications, we can find \textsl{substochastic} 
matrix $\PP'_{Z}$ fulfilling (\ref{eq:Siegmund_duality}). 
Then we add one extra state absorbing, say $\e_0$, and define $\PP_{Z}=\mathcal{F}_{\e_0}(\PP'_{Z})$. Note that then $\PP_{Z}$ fulfills 
(\ref{eq:Siegmund_duality}) for all states different from $\E$. This relation also implies that $\e_M$ is an absorbing state in
Siegmund dual, thus ${Z}$ has two absorbing states. Taking the limits as $n\to\infty$ on both sides of (\ref{eq:Siegmund_duality}), we have 
\begin{equation}\label{eq:siegm_pi}
\begin{array}{llllll}
\pi(\{\e_j\}^\downarrow) & =& \lim_{n\to\infty} \PP^n_{Z}(\e_j,\{\e_i\}^\uparrow) & =& P(\tau_{\e_M}<\tau_{\e_0} | Z_0=\e_j)=\rho(\e_j). 
\end{array}
\end{equation}
The stationary distribution of $X$ is related in this way to the absorption  of its Siegmund dual $Z$.
\medskip\par
It is convenient to define Siegmund duality in matrix form.
Let  $\C(\e_i,\e_j)=\mathbbm{1}(\e_i\preceq \e_j)$,
then the equality   (\ref{eq:Siegmund_duality})  can be expressed as
\begin{equation}\label{eq:Siegmund_duality_matrix}
 \PP_X^n\C=\C (\PP_{Z}'^{\textrm{ } n})^T
\end{equation}
Relation (\ref{eq:siegm_pi}) can be rewritten in matrix form as
\begin{equation*}\label{eq:siegm_pi_matrix}
\begin{array}{llllll}
\boldsymbol{\rho}=\boldsymbol{\pi} \C.
\end{array}
\end{equation*}
\noindent 
The inverse $\C^{-1}$ always exists, usually is denoted by $\mu$ and called \textsl{M\"obius function}.
To find a Siegmund dual it is enough to find $\PP_Z$ fulfilling (\ref{eq:Siegmund_duality_matrix}) 
with for $n=1$.

 Let $\preceq:=\leq$ be a total ordering on a finite state space $\E=\{1,\ldots,M\}$.
The chain $Y$ is \textbf{stochastically monotone} w.r.t to total ordering if $\forall i_1\leq i_2 \textrm{ }\forall j\textrm{ } \PP_Y(i_2,\{j\}^\downarrow) \leq \PP_Y(i_1,\{j\}^\downarrow)$.
We have 
\begin{lemma}[Siegmund \cite{Siegmund1976}]\label{lem:Siegmund_total}
Let $X$ be an ergodic Markov chain on $\E=\{1,\ldots,M\}$ with transition matrix $\PP_X$. Siegmund dual $Z$ (w.r.t. total ordering) exists if and only if
$X$ is stochastically monotone. In such a case $\PP_{Z}=\mathcal{F}(\PP'_{Z})$, where 
$$\PP'_{Z}(j,i)=\PP_X(i,\{j\}^\downarrow)-\PP_X(i+1,\{j\}^\downarrow)$$ for 
$i,j\in \E$ (we mean $\PP_X(i+1,\cdot)=0$).
\end{lemma}
\noindent 
Since the proof is one line long, we present it. 
\begin{proof}[Proof of Lemma \ref{lem:Siegmund_total}]
The main thing is to show that (\ref{eq:Siegmund_duality}) holds for $n=1$. We have
\begin{equation*}\label{eq:sieg_total}
\begin{array}{lll}
 \PP'_{Z}(j,i)&=&\PP'_{Z}(j,\{i\}^\uparrow)-\PP'_{Z}(j,\{i+1\}^\uparrow) = \PP_X(i,\{j\}^\downarrow)-\PP_X(i+1,\{j\}^\downarrow). 
\end{array}
\end{equation*}
The latter is non-negative if and only if $X$ is stochastically monotone.
\end{proof}
Let $X$ be and ergodic birth and death chain on $\E=\{1,\ldots,M\}$ with transition matrix 
\par 

\begin{equation}\label{eq:PX_bd} 
 \displaystyle \PP_{X}(i,i')= \left\{ 
 \begin{array}{llllllll}
  p'(i) & \textrm{if} & i'=i+1, \\[3pt]
  q'(i) & \textrm{if} & i'=i-1,  \\[3pt]
  1-(p'(i)+q'(i))& \textrm{if} & i'=i,\\[3pt]
 \end{array}
 \right.
\end{equation}
where $q'(1)=p'(M)=0$ and $p'(i)>0, i=1,\ldots, M-1,  q'(i)>0, i=2,\ldots,M$. 
Assume that $p'(i-1)+q'(i)\leq 1, i=2,\ldots,M$ (what is equivalent to stochastic monotonicity).

It is easily verifiable that when we rename transition probabilities: $p(i)=q'(i), q(i)=p'(i-1)$, then 
the transitions $\PP_{Y}$ defined in (\ref{eq:PZj})  are the transitions of Siegmund dual resulting from Lemma \ref{lem:Siegmund_total}.
From the known form of stationary distribution of an ergodic birth and death chain, and from relation (\ref{eq:siegm_pi}), it follows 
that for $\PP_{Y}$ given in (\ref{eq:PZj}) we have 
\begin{equation}\label{eq:1d_rho}
\rho(s)=\sum_{k\leq s} \pi(s)={\displaystyle\sum_{n=1}^s \prod_{r=1}^n \left({q(r)\over p(r)}\right) \over \displaystyle\sum_{n=1}^M \prod_{r=1}^n \left({q(r)\over p(r)}\right)}. 
\end{equation}

\subsection{Intertwinings between absorbing chains}\label{sec:intertwining}
Let $\Lambda$ be any nonsingular matrix of size $M\times M$.
We say that matrices $\PP_{X^*}$ and $\PP_{\hat X}$ of size $M\times M$ are \textbf{intertwined by a link} $\Lambda$ if
\begin{equation*}\label{eq:link_matrices}
\Lambda\PP_{X^*}=\PP_{\hat{X}}\Lambda.
\end{equation*}
Similarly, we say that vectors $ \hat{\nu}$ and $\nu^*$  of length $M$    are intertwined if 
\begin{equation}\label{eq:link_nus}
{\nu}^*=\hat{\nu}\Lambda
\end{equation}
\noindent
We say that link $\Lambda$ is $\e_M^*$-isolated if 
\begin{equation}\label{eq:Lambda_cond}
  \Lambda(\hat{\e},{\e}^*_M) = 
 \left\{ 
 \begin{array}{lll}
  \neq 0 & \mathrm{if\ } \hat{\e}=\hat{\e}_M,\\[10pt]
  0 & \mathrm{otherwise\ }.\\  
 \end{array}\right.
\end{equation}

\begin{lemma}\label{lem:duality_two_linksC}

 Let   $X^*$ and  $\hat X$ be Markov chains on $\E^*=\e_0^*\cup \hat{\E}$ and $\hat{\E}$ with transition matrices $\PP_{X^*}$ and $\PP_{\hat{X}}$ respectively.
 Moreover, assume $X^*$ has initial distribution $\nu^*$ and two absorbing states : $\e_0^*$ and $\e_M^*$, whereas $\hat{X}$ has one absorbing state $\hat{\e}_M$.
 Assume that $\PP_{X^*}'=\mathcal{F}^{-1}_{\e_0^*}(\PP_{X^*})$ and $\PP_{\hat{X}}$ are intertwined via $\e_M^*$-isolated link $\Lambda$.
 Let $\hat{\nu}=\nu^*\Lambda^{-1}$.
 Then we have 

$$pgf_{T^*_{\nu^*,\e_M^*}}(s) = \Lambda(\hat{\e}_M,\e_M^*)  \sum_{\hat{\e}\in \hat{\E}} \hat{\nu}(\hat{\e}) pgf_{\hat{T}_{\hat{\e},\hat{\e}_M}}(s).$$
 
\end{lemma}
\begin{proof}
$$\begin{array}{rll}
 P(T^*_{\nu^*,\e^*_M}\leq t)&=&\displaystyle P(X^*(t)=\e^*_M)=\sum_{\e^*\in \E^*\setminus\{\e_0^*\}}\nu^*(\e^*)\PP_{X^*}^t(\e^*,\e^*_M)\\[18pt]
&=&\displaystyle \sum_{\hat{\e}\in \E} \sum_{\e^*\in \E^*\setminus\{\e_0^*\}}\hat{\nu}(\hat{\e})\Lambda(\hat{\e},\e^*)\PP_{X^*}^t(\e^*,\e^*_M) \\[18pt]
&=&\displaystyle\sum_{\hat{\e}\in \E} \sum_{\hat{\e}_2\in \hat{\E}}\hat{\nu}(\hat{\e}) \PP_{\hat{X}}^t(\hat{\e},\hat{\e}_2)\Lambda(\hat{\e}_2,\e^*_M)\\[18pt]
&=&\displaystyle  \Lambda(\hat{\e}_M,\e_M^*)\sum_{\hat{\e}\in \E}  \hat{\nu}(\hat{\e}) \PP_{\hat{X}}^t(\hat{\e},\hat{\e}_M) 
\end{array}
$$
Now, for $pgf$ we have: \smallskip\par\noindent
$pgf_{T^*_{\nu^*,\e^*_M}}(s)=$
$$\begin{array}{rll}
& &\displaystyle \sum_{k=0}^{\infty} P(T^*_{\nu^*,\e^*_M}=k)s^k= \sum_{k=0}^{\infty} \left(P(T^*_{\nu^*,\e^*_M}\leq k)-P(T^*_{\nu^*,\e^*_M}\leq k-1) \right) s^k\\[16pt]
&=& \displaystyle \Lambda(\hat{\e}_M,\e_M^*)\sum_{k=0}^{\infty} \left(
 \sum_{\hat{\e}\in \hat{\E}} \hat{\nu}(\hat{\e}) \PP_{\hat{X}}^k(\hat{\e},\hat{\e}_M)
 -
  \sum_{\hat{\e}\in \hat{\E}} \hat{\nu}(\hat{\e}) \PP_{\hat{X}}^{k-1}(\hat{\e},\hat{\e}_M)
   \right) s^k \\[16pt]
&=&\displaystyle  \Lambda(\hat{\e}_M,\e_M^*)\sum_{\hat{\e}\in \hat{\E}} \hat{\nu}(\hat{\e}) \sum_{k=0}^{\infty} \left(
  \PP_{\hat{X}}^k(\hat{\e},\hat{\e}_M)
 -
  \PP_{\hat{X}}^{k-1}(\hat{\e},\hat{\e}_M)
   \right) s^k \\[16pt]
&=&\displaystyle  \Lambda(\hat{\e}_M,\e_M^*)\sum_{\hat{\e}\in \hat{\E}} \hat{\nu}(\hat{\e}) \sum_{k=0}^{\infty} \left(
  P(\hat{T}_{\hat{\e},\hat{\e}_M}\leq k)
 -
  P(\hat{T}_{\hat{\e},\hat{\e}_M}\leq k-1)
   \right) s^k\\[16pt]
&=&\displaystyle  \Lambda(\hat{\e}_M,\e_M^*)\sum_{\hat{\e}\in \hat{\E}} \hat{\nu}(\hat{\e}) \sum_{k=0}^{\infty}
  P(\hat{T}_{\hat{\e},\hat{\e}_M}=k) s^k= \Lambda(\hat{\e}_M,\e_M^*)\sum_{\hat{\e}\in \hat{\E}} \hat{\nu}(\hat{\e}) pgf_{\hat{T}_{\hat{\e},\hat{\e}_M}}(s).
\end{array}
$$  
\end{proof}

\begin{corollary}
 Let assumptions of Lemma \ref{lem:duality_two_linksC} hold and, in addition, let $\hat{\nu}$ be a distribution. Then, we have 
 $$ {T}_{\nu^*,{\e_M^*}}^* = \Lambda(\hat{\e}_M,\e_M^*) \hat{T}_{\hat{\nu},\hat{\e}_M}.$$
\end{corollary}


\par 
\noindent
From Fill, Lyzinski \cite{FilLyz14} we can deduce the following.
\begin{lemma}\label{lem:spectral_poly}

 Let   $X^*$ be a birth and death chain on $\E^*=\{0,\ldots,M\}$ 
   with transition matrix $\PP_{X^*}$ given in (\ref{eq:PZj}) 
 with two absorbing   states: $0$ and $M$.
 Let $\PP_{X^*}'=\mathcal{F}^{-1}_{0}(\PP_{X^*})$. Assume the eigenvalues of $\PP_{X^*}'$ are non-negative,
 denote them by $0\leq \lambda_1\leq \ldots\leq \lambda_M=1$. 
 
  Define $\Q_1:=\I$ and
 $$\mathbf{Q}_k:={(\PP_{X^*}'-\lambda_1\I)\cdots(\PP_{X^*}'-\lambda_{k-1}\I)\over (1-\lambda_1)\cdots(1-\lambda_{k-1})}, k=2,\ldots,M$$
 
 Let ${\Lambda}$ be the lower triangular square matrix of size $M\times M$ defined as 
 \begin{equation}\label{eq:link_spectral}
  {\Lambda}(k,\cdot)=\mathbf{Q}_k(1,\cdot), \quad k=1,\ldots,M.
 \end{equation}
 Then, $\PP_{X^*}'$ and $\PP_{\hat{X}}$ are intertwined via link $\Lambda,$ where 
 \begin{equation}\label{eq:pure_birth2}
   {\PP}_{\hat{X}}(i,i')=\left\{
 \begin{array}{llll}  
   1-\lambda_i & \mathrm{if }\  i'=i+1, \\
  \lambda_i & \mathrm{if} \ i'=i, \\
  0 & \mathrm{otherwise}. \\
 \end{array}\right.
 \end{equation}
 is a   matrix of size $M\times M$.
 \end{lemma}
 
Note that Lemma \ref{lem:spectral_poly} is similar to Theorem 4.2 in \cite{FilLyz14}, the difference is that
therein $\Lambda$ is a stochastic matrix, whereas in Lemma \ref{lem:spectral_poly} it can be substochastic
(it is strictly substochastic if $q(0)>0$).
Almost identical $\Lambda$ was considered in \cite{Gong2012}, their Proposition 3.3 yields.
\begin{lemma}\label{lem:Lambda_rho} \ \par 
 \begin{itemize} 
  \item The matrices $\Q_k, 1,\ldots,M$ are non-negative and substochastic.
  \item The matrix $\Lambda$ is non-negative and substochastic, it is lower triangular and 
  $$\Lambda(1,1)=1, \qquad \Lambda(M,M)=\rho(1),$$
  thus $\Lambda$ is nonsingular.
 \end{itemize}

\end{lemma}

 \begin{remark}\rm 
 Note that in case $X^*$ has no transition to 0, \ie $q(1)=0$,  it is actually a chain on $\{1,\ldots,M\}$ and 
  $\PP_{X^*}'=\mathcal{F}^{-1}_{0}(\PP_{X^*})$ is a stochastic matrix. Then $\Lambda$ is a stochastic matrix and $\Lambda(M,M)=1$.
 
\end{remark}

\section{Proofs}

\subsection{Properties of Kronecker product}
In this section we recall some useful properties of Kronecker product
and  formulate lemma relating   eigenvectors and eigenvalues of some combination of Kronecker products.
\smallskip\par 

We will exploit the following properties 
\begin{itemize}
\item bilinearity:
\begin{equationprop}\label{prop:blin}
 \A\otimes (\B+\C)=\A\otimes\B +\A\otimes\C,
\end{equationprop}
\item mixed product:
\begin{equationprop}\label{prop:prod}
(\A\otimes \B)(\C \otimes \D)=(\A \C)\otimes(\B \D),
\end{equationprop}
\item inverse and transposition:
\begin{equationprop}\label{prop:inv}
(\A\otimes \B)^{-1}=(\A)^{-1}\otimes(\B)^{-1},
\end{equationprop}
\vspace{-0.3cm}
\begin{equationprop}\label{prop:transp}
(\A\otimes \B)^T=(\A)^T\otimes(\B)^T.
\end{equationprop}

\item  eigenvalue and eigenvector:
\begin{equationprop}\label{prop:eigen}
\begin{array}{l}
  \textrm{having eigenvalues } \alpha_j \textrm{with corresponding left eigenvectors } a_j \textrm{ for each}\\
  \textrm{matrix } \A_j,  j\leq n,   \textrm{we note that }  \sum_{j\leq n} \alpha_j \textrm{ with  } \bigotimes_{j\leq n} a_j \textrm{ and } \prod_{j\leq n} \alpha_j\\
   \textrm{with } \bigotimes_{j\leq n} a_j \textrm{ are eigenvalue } 
  \textrm{and left eigenvector of }
  \A =\bigoplus_{j\leq n} \A_j  \\ 
  \textrm{and } \A' =\bigotimes_{ j \leq n} \A_j\textrm{ respectively}.
\end{array}
\end{equationprop}

%
%
\end{itemize}

Last property leads us to the following lemma.
\begin{lemma}\label{lem:kron_eigen}
For all $1 \leq j\leq n$ and $1 \leq i \leq m$, let $a_j$ be the left eigenvectors with the corresponding eigenvalues $\alpha_j$
of square matrices $\A_j^{(i)}$ of size $k_j$ respectively. \\
Let  $\B_i, i=1,\ldots,m$ be square matrices of size $\prod_{j=1}^n k_j$ such that $\sum_{i=1}^m \B_i =\I$,
where $\I$ is identity of size $\prod_{j=1}^n k_j$. Then $ \prod_{j\leq n} \alpha_j$ with $ \bigotimes_{j\leq n} a_j$ are
the eigenvalue and the left eigenvector of $\A =\sum_{i=1}^m (\bigotimes_{j\leq n} \A_j^{(i)})\B_i$. \\
Similarly, if  $b_i, i=1,\ldots,m$  are real numbers such that $\sum_{i=1}^m b_i$ we have that 
$\bigotimes_{j\leq n} a_j$ is the left eigenvector with the corresponding eigenvalue 
$ \prod_{j\leq n} \alpha_j$  of the matrix   $\A =\sum_{i=1}^m (\bigotimes_{j\leq n} \A_j^{(i)})b_i$.
\end{lemma}
\begin{proof}
We have
$$\begin{array}{rllll}
\displaystyle\bigotimes_{j\leq n} a_j \sum_{i=1}^m (\bigotimes_{j\leq n} \A_j^{(i)})\B_i&=&\displaystyle \sum_{i=1}^m \bigotimes_{j\leq n} a_j  (\bigotimes_{j\leq n} \A_j^{(i)})\B_i = \sum_{i=1}^m \bigotimes_{j\leq n} (a_j   \A_j^{(i)})\B_i\\
&=&\displaystyle  
\sum_{i=1}^m \bigotimes_{j\leq n} (a_j   \alpha_j)\B_i= 
\sum_{i=1}^m  \prod_{j\leq n} \alpha_j\bigotimes_{j\leq n} a_j   \B_i \\
&=&  \displaystyle
\prod_{j\leq n} \alpha_j\bigotimes_{j\leq n} a_j  \sum_{i=1}^m    \B_i
= \displaystyle \prod_{j\leq n} \alpha_j\bigotimes_{j\leq n} a_j .\\ 
\end{array}
$$
Similarly,
$$\begin{array}{rllll}
\displaystyle \bigotimes_{j\leq n} a_j \sum_{i=1}^m (\bigotimes_{j\leq n} \A_j^{(i)})b_i &=&\displaystyle \sum_{i=1}^m \bigotimes_{j\leq n} a_j  (\bigotimes_{j\leq n} \A_j^{(i)})b_i = \sum_{i=1}^m \bigotimes_{j\leq n} (a_j   \A_j^{(i)})b_i\\
&=& \displaystyle
\sum_{i=1}^m \bigotimes_{j\leq n} (a_j   \alpha_j)b_i= 
\sum_{i=1}^m  \prod_{j\leq n} \alpha_j\bigotimes_{j\leq n} a_j   b_i \\
&=&\displaystyle 
\prod_{j\leq n} \alpha_j\bigotimes_{j\leq n} a_j  \sum_{i=1}^m    b_i
= \prod_{j\leq n} \alpha_j\bigotimes_{j\leq n} a_j .
\end{array}
$$
\end{proof}

Substitution of stochastic matrices $\PP_j^{(i)}$ with stationary distributions 
$\pi_j$ (for all $1\leq j\leq n$, $1\leq i\leq m$) to matrices $\A_j^{(i)}$ with left eigenvectors 
$\alpha_j$ (for all $1\leq j\leq n$, $1\leq i\leq m$) gives us following Corollary 
(keeping in mind that 1 is the eigenvalue corresponding to the eigenvector being the stationary distribution):

\begin{corollary}\label{cor:stat_dist}
Let $\PP_j^{(i)}$ be a stochastic matrix of size $k_j$ with $\pi_j$ the stationary distribution for all $1 \leq j \leq n, 1 \leq i \leq m$. 
Let $\B_i$, $1\leq i \leq m$ be square matrices of size $\prod_{j=1}^n k_j$ such that $\sum_{i=1}^m \B_i =\I$,
where $\I$ is the identity of size $\prod_{j=1}^n k_j$.
Similarly, if $b_i$, $1\leq i \leq m$ are real numbers such that $\sum_{i=1}^m b_i=1$,
then the stochastic  matrices of the form $\sum_{i=1}^m (\bigotimes_{j\leq n} \PP_j^{(i)})\B_i$ or $\sum_{i=1}^m (\bigotimes_{j\leq n} \PP_j^{(i)})b_i$ have stationary distribution of the form $ \bigotimes_{j\leq n} \pi_j$. 
\end{corollary}

\subsection{Proof of Theorem \ref{thm:main_rho}}\label{sec:proof_rho}

We will   find an ergodic Markov chain $X$ with transition matrix $\PP_X$ and some partial ordering of the state space (expressed
by an ordering matrix $\C$) and show  that (\ref{eq:siegm_pi}) is equivalent to (\ref{eq:rho_product}).

\par 
Let $\PP_{Z_j}^{(k)}$ (on $\E_j=\{0,\ldots,N_j\}$) be as in theorem. Let $X_j^{(k)}$ be 
  ergodic chains on $\E'_j=\{1,\ldots,N_j\}$ with transition matrix $\PP_{X_j}^{(k)}$, such that $Z_j^{(k)}$
is its Siegmund dual w.r.t. total ordering.
I.e., let $\C_j(s,t)=\mathbbm{1}(s\leq t)$, and duality means that 
$$\PP_{X_j}^{(k)}\C_j=\C_j (\PP_{Z_j'}^{(k)})^T,$$
where $\PP_{Z_j'}^{(k)}=\mathcal{F}^{-1}_0(\PP_{Z_j}^{(k)})$.
Assumption (\ref{eq:rho_k}) and relation (\ref{eq:siegm_pi}) imply that for fixed $j$,
the chains $X_j^{(k)}, k=1,\ldots,m$ have the same stationary distribution, denote it by  $\pi_j$.
The relation (\ref{eq:siegm_pi}) means that $\boldsymbol{\rho}_j=\boldsymbol{\pi}_j \C_j$. 
On the state space $\E=\bigotimes_{j\leq d} \E_j$ let us introduce the ordering expressed by matrix $\C=\bigotimes_{j\leq d} \C_j$.
From (\ref{eq:Siegmund_duality_matrix}) we can calculate the matrix $\PP_X$:

\begin{equation*}
\begin{array}{lllll}
 \PP_X &=&  \C\PP_Z^T\C^{-1} = \displaystyle (\bigotimes_{j\leq d} \C_j)\left(\sum_{k=1}^m \B_k(\bigotimes_{j\leq d} \R_{Z_j'}^{(k)})\right)^T(\bigotimes_{j\leq d} \C_j)^{-1} \\[16pt]
       &\stackrel{\ref{prop:transp},\ref{prop:inv}}{=}& \displaystyle (\bigotimes_{j\leq d} \C_j)\left(\sum_{k=1}^m (\bigotimes_{j\leq d} (\R_{Z_j'}^{(k)})^T)\B_k^T\right)(\bigotimes_{j\leq d} \C_j^{-1})\\[16pt]
       &\stackrel{\ref{prop:blin}}{=}& \displaystyle \sum_{k=1}^m(\bigotimes_{j\leq d} \C_j) \left(\bigotimes_{j\leq d} (\R_{Z_j'}^{(k)})^T\right)\B_k^T(\bigotimes_{j\leq d} \C_j^{-1})\\[16pt]
       &\stackrel{\ref{prop:prod}}{=}& \displaystyle \sum_{k=1}^m(\bigotimes_{j\leq d} \C_j (\R_{Z_j'}^{(k)})^T \C_j^{-1})(\bigotimes_{j\leq d} \C_j) \B_k^T(\bigotimes_{j\leq d} \C_j^{-1}). 
\end{array}
\end{equation*}
Let us define
$$\R_{X_j}^{(k)}=\C_j (\R_{Z_j'}^{(k)})^T \C_j^{-1} =
\left\{ 
\begin{array}{llll}
 \PP_{X_j}^{(k)} & \textrm{if } j\in A_k,  \\
 \I_j & \textrm{if } j\notin A_k.  \\
\end{array}
\right.$$
In the case $j\in A_k$, the distribution $\pi_j$ is the unique stationary distribution.
In the case $j \notin A_k$, any distribution is an invariant measure, however, we fix it to be $\pi_j$.

We have 
$$\PP_X= \sum_{k=1}^m(\bigotimes_{j\leq d} \R_{X_j}^{(k)})(\bigotimes_{j\leq d} \C_j) \B_k^T(\bigotimes_{j\leq d} \C_j^{-1}).$$
From property \ref{prop:blin} we have that 
$$\sum_{k=1}^m (\bigotimes_{j\leq n} \C_j) \B_k^T(\bigotimes_{j\leq n} \C_j^{-1})= (\bigotimes_{j\leq n} \C_j)\sum_{k=1}^m \B_k^T(\bigotimes_{j\leq n} \C_j^{-1})= (\bigotimes_{j\leq n} \C_j)(\bigotimes_{j\leq n} \C_j^{-1})=\I,$$
thus Corollary \ref{cor:stat_dist} implies that $\boldsymbol{\pi}=\bigotimes_{j\leq d} \boldsymbol{\pi}_j$
  is  the stationary distribution of $\PP_X$, 
thus $\boldsymbol{\rho}=\boldsymbol{\pi}\C,$ what is equivalent to (\ref{eq:rho_product}).
\par 
\begin{flushright}
 $\square$
\end{flushright}

\subsection{Proof of Theorem \ref{thm:main_abs_times}}\label{sec:proof_abs_times}

To prove the theorem we will construct an $\N$-isolated  link $\Lambda$, so that 
$\PP_{X^*}'$ and $\PP_{\hat{X}}$, given in (\ref{eq:PZ_general2}) and (\ref{eq:abs_time_res_prod}) respectively, are intertwined via this link.
\par \noindent
Consider matrix $\PP_{X^*_j}'$. Define stochastic matrix $\PP_{\hat{X}_j}$ of size $N_j\times N_j$ defined as:
 $$
 {\PP}_{\hat{X}_j}(i,i')=\left\{
 \begin{array}{llll}  
   1-\lambda^{(j)}_i & \mathrm{if }\  i'=i+1, \\
  \lambda_i^{(j)} & \mathrm{if} \ i'=i, \\
  0 & \mathrm{otherwise}. \\
 \end{array}\right.
 $$
 Let $\Lambda_j$ be a link intertwining matrices $\PP_{X^*_j}'$ and ${\PP}_{\hat{X}_j}$ given in  (\ref{eq:link_spectral}).
 Define 
\begin{equation*}
 \R_{\hat{X}_j^{(k)}} = 
 \left\{ 
 \begin{array}{lll}
   \PP_{\hat{X}_j}  & \mathrm{if\ } j\in A_k,\\[10pt]
  \I_j & \mathrm{if\ } j\notin A_k.\\  
 \end{array}\right.
\end{equation*}
Note that matrices $\mathbf{R}_{X^*_j}'$ and $\R_{\hat{X}_j^{(k)}} $ are also intertwined via $\Lambda_j$ for any $j=1,\ldots,d$ and
any $k=1,\ldots,m$. Any link intertwines two identity matrices, which is the case for  $j\notin A$. I.e.,
we have $\Lambda_j\mathbf{R}_{X^*_j}'=\R_{\hat{X}_j^{(k)}}\Lambda_j, j=1,\ldots,d$.
 Define 
 $$ \Lambda=\bigotimes_{j\leq d} \Lambda_j.$$
  We have 
  $$
  \begin{array}{rlllll}
 \Lambda \PP_{X^*}' &=&\displaystyle \bigotimes_{j\leq d} \Lambda_j   \sum_{k=1}^m b_k\left(\bigotimes_{j\leq d} \R_{X_j^{*(k)}}'\right) =\sum_{k=1}^m  b_k  \left(\bigotimes_{j\leq d} \Lambda_j \R_{X_j^{*(k)}}'\right) \\
 &=&\displaystyle   
 \sum_{k=1}^m  b_k  \left(\bigotimes_{j\leq d} \R_{\hat{X}_j^{(k)}}\Lambda_j\right)= \sum_{k=1}^m  b_k  \left(\bigotimes_{j\leq d} \R_{\hat{X}_j^{(k)}}\right) \bigotimes_{j\leq d}\Lambda_j\\  
  \end{array}
  $$
 \noindent 
 Simple calculations yield that $\PP_{\hat{X}}$ given in (\ref{eq:abs_time_res_prod}) can be written as 
 $\displaystyle\sum_{k=1}^m  b_k  \left(\bigotimes_{j\leq d} \R_{\hat{X}_j^{(k)}}\right)$. Thus, we have 
 $\Lambda \PP_{X^*}'=\PP_{\hat{X}}\Lambda$.
 Now, let us 
  calculate $\hat{\nu}=\nu^*\Lambda^{-1}$
 (note that $\Lambda$ is nonsingular because of property (\ref{prop:inv}) and fact that each $\Lambda_j, j=1,\ldots,d$ and identity matrices $\mathbf{I}_j$ are 
 nonsingular). In other words, we have $\nu^*=\hat{\nu}\Lambda$.
Equation (\ref{eq:vg_b}) holds, since $\Lambda$ is lower triangular.

 Moreover, $\Lambda$ is $(N_1,\ldots,N_d)$-isolated, since we have 
 $$\Lambda((i_1,\ldots,i_d),\N)=\prod_{j=1}^d \Lambda(i_j,N_j)\stackrel{(*)}{=}
 \left\{ 
 \begin{array}{lll}
 \prod_{j=1}^d \rho_j(1) & \mathrm{if\ } i_j=N_j \mathrm{\ for\ all\ j=1,\ldots,d},\\[10pt]
  0 & \mathrm{otherwise\ },\\  
 \end{array}\right. $$ 
 where in $\stackrel{(*)}{=}$ we used Lemma  \ref{lem:Lambda_rho}.
 Applying Lemma \ref{lem:duality_two_linksC} completes the proof.

\begin{flushright}
 $\square$
\end{flushright}

\section{Outline of alternative proof of Theorem \ref{thm:main_abs_times}: strong stationary duality approach}\label{sec:alt}
In Theorem \ref{thm:main_abs_times} we related absorption time $T^*_{\nu^*,\N}$ of $X^*$ with absorption time $\hat{T}_{\hat{\nu},\N}$ 
of $\hat{X}$. This was done by finding a specific matrix $\Lambda$, such that  $\Lambda\PP_{X^*}=\PP_{\hat{X}}\Lambda$,
exploiting existence of such $\Lambda$ for $X^*$ and $\hat{X}$ being birth and death chains. 
The exploited $\Lambda$ is related to spectral polynomials of the stochastic matrix $\PP_{X^*}$. 
Such link appeared first naturally as a link between an  ergodic chain $X$ and an absorbing chain $X^*$.
The proof of Theorem \ref{thm:main_abs_times} in case $q_j(1)=0$ (\ie Corollary \ref{cor:abc} a)) 
can be different, using intermediate ergodic chain. In this section we will describe its outline.

\medskip 

\paragraph{Strong stationary duality}
Let $X$ be an ergodic Markov chain on $\E=\{\e_1,\ldots,\e_N\}$ with initial distribution  $\nu$ and transition 
matrix $\PP_X$.
Let ${\E}^*=\{ {\e_1^*},\ldots,{\e_M^*}\}$ be the, possibly different, state space of the absorbing Markov chain $X^*$,
with transition matrix ${\PP_{X^*}}$
and initial distribution ${\nu^*}$, whose unique absorbing state is denoted by $ {\e_M^*}$. 
Assume that  $\Lambda^*$  is a stochastic $M\times N$ matrix 
 satisfying $\Lambda({\e_M^*},\e)=\pi(\e)$. We say that $X^*$ is 
a \textsl{strong stationary dual} (SSD) of $X$ with link $\Lambda^*$ if
\begin{equation}\label{eq:duality}
\nu={\nu^*}\Lambda \quad \textrm{and} \quad \Lambda^*\PP_X=\PP_{X^*}\Lambda^*.
\end{equation}
Diaconis and Fill  \cite{Diaconis1990a} prove that then the absorption time ${T^*}$ of $X^*$ is the so called \textsl{strong stationary time} for $X$.
This is a random variable $T$ such that $X_T$ has distribution $\pi$ and $T$ is independent of $X_T$. 
The main application is to studying the rate of convergence of an ergodic chain to its stationary distribution, since for such a random variable we always have $d_{TV}(\nu\PP_X^k,\pi)\leq sep(\nu\PP_X^k,\pi)\leq P(T>k)$,
where $d_{TV}$ stands for the \textsl{total variation distance}, and $sep$ stands for the \textsl{separation `distance'}. For details, see Diaconis and Fill  \cite{Diaconis1990a}.
We say that SSD is \textbf{sharp} if ${T^*}$ corresponds to stochastically the smallest SST, then we have $sep(\nu\PP_X^k,\pi)=P( {T^*}>k)$,
the corresponding SST $T^*$ is often called \textbf{fastest strong stationary time (FSST)}.
\smallskip\par 
\paragraph{Strong stationary duality for birth and death chain} 
Let $X$ be an ergodic birth and death chain on $\E=\{1,\ldots,M\}$, whose time reversal is stochastically monotone.
With transitions given in (\ref{eq:PX_bd}). Diaconis and Fill \cite{Diaconis1990a} show that 
an absorbing birth and death chain $X^*$ on $\E^*=\E$ with transitions given by
 $$
 \begin{array}{llll}
  \PP_{X^*}(i,i-1)  & = & {H(i-1)\over H(i)} p'(i), \\[10pt]
  \PP_{X^*}(i,i+1)  & = & {H(i+1)\over H(i)} q'(i+1), \\[10pt]
  \PP_{X^*}(i,i)  & = & 1-(p'(i)+q'(i+1)), \\
 \end{array}
 $$
 is a \textsl{sharp} SSD for $X$. Here we have 
 \begin{equation}\label{eq:link_classic}
 H(j)=\sum_{k\leq j}\pi(k), \qquad \Lambda^*(i,j)=\mathbbm{1}\{i\leq j\}{\pi(i)\over H(j)}. 
 \end{equation}
 
 Moreover, starting from an absorbing birth and death chain $X^*$ on $\E=\{1,\ldots,M\}$,
 whose unique absorbing state is $M$, Theorem 3.1 in \cite{Fill2009} states that we can find 
 an ergodic chain $X$ (and some stationary distribution $\pi$), such that $X^*$ is  its \textsl{sharp} SSD 
 with link given in (\ref{eq:link_classic}).

 \paragraph{Spectral pure-birth chain}
 Again, let $X$ be an ergodic birth and death chain on $\E=\{1,\ldots,M\}$. 
 Assume its eigenvalues are non-negative, $0\leq \lambda_1\leq\ldots\leq \lambda_M=1$.
 Then, the chain $\hat{X}$ with transitions given in (\ref{eq:pure_birth2}) is its \textsl{sharp} 
 SSD with link $\hat{\Lambda}$ given in (\ref{eq:link_spectral}).
 
 \paragraph{The outline of an alternative proof}
 As in Section \ref{sec:proof_abs_times}, the main idea is to show that two absorbing birth and death chains $X^*_j$ and $\hat{X}_j$ (pure-birth) on $\E_j=\{1,\ldots,N_j\}$ 
 are intertwined by an $N_j$-isolated link $\Lambda_j$. Collecting above findings, we have (skipping conditions on initial distributions):
 \begin{itemize}
  \item Let $X_j$ be an ergodic chain on $\E_j$, whose $X^*_j$ is a sharp SSD, \ie we have $\Lambda_j^*\PP_{X_j}=\PP_{X_j^*}\Lambda_j^*$.
  \item Let $\hat{X}_j$ be a pure-birth sharp SSD for $X_j$, \ie we have $\hat{\Lambda}_j\PP_X=\PP_{\hat{X}_j}\hat{\Lambda}_j$.
 \end{itemize}
It means that absorption times $T^*$ and $\hat{T}$ are equal in distribution (since both $X^*$ and $\hat{X}$ are 
\textsl{sharp} SSDs of $X$. Mathematically, we have 
$$ \Lambda_j\PP_{X^*_j}=\PP_{\hat{X}_j}\Lambda_j, \qquad \textrm{where}\quad \Lambda_j=\hat{\Lambda}_j\left(\Lambda_j^{*}\right)^{-1},$$
\ie ${X^*_j}$ and $\hat{X}_j$ are intertwined by the link $\Lambda_j$, which is $N_j$-isolated.
Intertwining between two absorbing birth and death chains via an ergodic chain is depicted in Fig. \ref{fig:intertw}. 
Taking $\Lambda=\bigotimes_{j\leq d} \Lambda_j$ and $\hat{\nu}=\nu^*\Lambda^{-1}$ we proceed with the proof of 
Theorem \ref{thm:main_abs_times} as in Section \ref{sec:proof_abs_times}.

\tikzstyle{branch}=[fill,shape=circle,minimum size=3pt,inner sep=0pt]

\begin{figure}
\begin{center}
\begin{tikzpicture}[scale=0.8,every node/.style={scale=1}]

 

   \node at (-2.5,0) (nodePg) {$X^*$};
   \node at (0,4.5) (nodeP) {$X$};
   \node at (2.5,0) (nodePh) {$\hat{X}$};
      \draw [->, scale=0.4] (nodePg) -- (nodeP) node [midway, left=0.1cm] (TextNode1) {$\Lambda^*$}; 
      \draw [->, scale=0.4] (nodePh) -- (nodeP) node [midway, right=0.1cm] (TextNode2) {$\hat{\Lambda}$};
      \draw [->, scale=0.4] (nodePg) -- (nodePh); 
      
      \node at (0,-0.5) (nodeLambda) {$\Lambda=\Lambda^* \hat{\Lambda}^{-1}$};
      
\end{tikzpicture}
\end{center}\caption{Intertwining between absorbing chains $X^*$ and $\hat{X}$ via ergodic chain $X$.}
\end{figure}
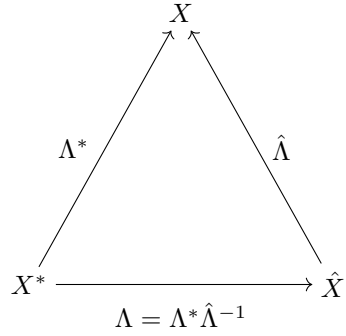\label{fig:intertw}

\section{Examples}
\subsection{One dimensional gambler's ruin problem with $N=3$: calculating $T^*_{2,3}$}\label{sec:example_a}

Here we present a simple example for calculating $T^*_{2,3}$ in a one dimensional gambler's ruin problem using 
Theorem \ref{thm:main_abs_times}. We also check that calculations agree with formula  (\ref{eq:gTi}).
 \begin{example}
 Let $d=1, N_1=3$ and $p_1(1)=p_1(2)=p>0, q_1(1)=q_1(2)=q>0$ such that $p\neq q$ and $p+q+\sqrt{pq}<1$. 
 The transition matrix of $\PP_{X^*_1}$ is
 following
 $$\PP_{X^*_1}=\left(
 \begin{array}{cccc}
  1 & 0 & 0 & 0 \\[3pt]
  q & 1-p-q & p & 0 \\[3pt]
  0 & q & 1-p-q & p  \\[3pt]
  0 & 0 & 0 & 1 \\  
 \end{array}\right).
$$
Then, the $pgf$ of time to absorption starting at 2 is given by:
\begin{equation}\label{eq:pgfT23}
 pgf_{T^*_{2,3}}(s)={p(q+p+\sqrt{qp})(-q-p+\sqrt{qp})u(1-u(1-q-p))
\over (p^2+qp+q^2)(1-u(1-q-p-\sqrt{qp})(-1+u(1-q-p+\sqrt{qp})}
\end{equation}
\end{example}

\begin{proof}

 We have $\PP_{X^*}=\PP_{X^*_1}$.
The eigenvalues of $\PP_{X^*}'=\mathcal{F}_0^{-1}(\PP_{X^*})$ are $\lambda_1=1-p-q-\sqrt{pq}, \lambda_2=1-p-q+\sqrt{pq}, \lambda_3=1$.
The transitions of $\PP_{\hat{X}}$ are following
 $$\PP_{\hat{X}}=\left(
 \begin{array}{cccc}
  \lambda_1 & 1-\lambda_1 &  0 \\[3pt]
  0 & \lambda_2 & 1-\lambda_2  \\[3pt]
  0 & 0 & 1 \\  
 \end{array}\right).
$$
Thus, 
$$pgf_{\hat{T}_{1,3}}(s)={(1-\lambda_1)(1-\lambda_2)s^2\over (1-\lambda_1 s)(1-\lambda_2 s)},\quad
pgf_{\hat{T}_{2,3}}(s)={(1-\lambda_2)s\over (1-\lambda_2 s)}.$$
Calculating $\Lambda$ from (\ref{eq:link_spectral})  (for $\PP_{X^*}$) we obtain
  $$\Lambda=\left(
 \begin{array}{cccc}
  1 & 0 & 0  \\[3pt]
  {\sqrt{pq}\over q+p+\sqrt{qp}} & {p\over q+p+\sqrt{qp}}  & 0  \\[3pt]
  0 & 0  & \rho(1) \\  
 \end{array}\right).
$$
Calculations yield (we have $\nu^*(2)=1$) $$\hat{\nu}=\nu^*\Lambda^{-1}=\left(-\sqrt{q\over p}, 1+{q\over p}+\sqrt{q\over p}\right).$$
From (\ref{eq:gTi})  we have $\rho(i)={1-\left({q\over p}\right)^i\over 1-\left({q\over p}\right)^3}, i=1,2,3$.
Finally,
$$pgf_{T^*_{2,3}}(s)=\rho(1) \left( -\sqrt{q\over p} pgf_{\hat{T}_{1,3}}(s) +\left(1+{q\over p}+\sqrt{q\over p}\right)pgf_{\hat{T}_{2,3}}(s) \right),$$
what can be written as  (\ref{eq:pgfT23}).
On the other hand, (\ref{eq:pgf_TsN_j}) states that 
$$pgf_{T^*_{2,3}}(s)=\rho(2) { {(1-\lambda_1)(1-\lambda_2)s^2\over (1-\lambda_1 s)(1-\lambda_2 s)} \over 
{(1-\lambda_1^{\floor{2}})s \over 1-\lambda_1^{\floor{2}}s }
} = { 1-\left({q\over p}\right)^2\over 1-\left({q\over p}\right)^3} \cdot {(1-\lambda_1)(1-\lambda_2)s^2 (1-\lambda_1^{\floor{2}}s)\over (1-\lambda_1 s)(1-\lambda_2 s)(1-\lambda_1^{\floor{2}})s} 
 ,$$
where $\lambda_1^{\floor{2}}=1-(p+q)$, which, as can be checked, is equivalent to (\ref{eq:pgfT23}).
 \end{proof}

\subsection{Winning probabilities and absorption time: changing $r$ coordinates at one step in $d$-dimensional game}\label{sec:example_b}
We will present an example for both Theorems, \ref{thm:main_rho} and \ref{thm:main_abs_times}.
The chains $\PP_{Z^{(k)}_j}$ in Theorem \ref{thm:main_rho} are  quite general, but in this 
example we consider birth and death chains \ie \ we will use
$\PP_{Z^{(k)}_j}=\PP_{X^*_j}$ from Theorem \ref{thm:main_abs_times} (birth and death chains given in  (\ref{eq:PZj})).
Similarly, we have $\R_{Z^{(k)}_j}'=\R_{X^{*(k)}_j}'$ and $\PP_{Z}=\PP_{X^*}$.
\par 
\begin{example}
\rm
The idea of the example is the following. We construct $d$-dimensional game from one-dimensional games,
in such a way, that at one step we play with $r$ other players, where $r\in\{1,\ldots,d\}$. 
In other words, the multidimensional chain can change at most $r$ coordinates in one step.

\par 
Moreover we will take, as $\B_i:=b_i$ real numbers. 
In both theorems let us take $0<r<d$, $m={{d}\choose{r}}+1$ and 
$b_k=1, k=1,\ldots,m-1, b_{m}=1-{{d}\choose{r}}.$ Let us enumerate combinations of $r$ positive numbers no greater than 
$d$  in some way (see e.g.,\cite{Mud65}). Let $\A_k$ be k-th combination from this numbering, for $k=1,\ldots,m-1$ and $\A_{m}=\emptyset$. 
Then we have 
$$  \PP'_{Z}=\sum_{k=1}^{m} \B_k\left(\bigotimes_{j\leq d} \R'_{Z_j^{(k)}}\right)=\sum_{k=1}^{{{d}\choose{r}}} \left(\bigotimes_{j\leq d} \R'_{Z_j^{(k)}}\right)-\left({{d}\choose{r}}-1\right)\bigotimes_{j\leq d} \I_j.$$
We have that $\R'_{Z_j^{(k)}}=\PP'_{Z_j}$ if $\{j\}\in \A_k$ and $\R'_{Z_j^{(k)}}=\I_j$ otherwise 
(for $\{j\}\not\in \A_k$), thus $\PP'_{Z}=$
$$\sum_{1\leq i_1<i_2< \ldots <i_r \leq d}(\bigotimes_{j< i_1} \I_j)\otimes \PP'_{Z_{i_1}} \otimes(\bigotimes_{i_1<j<i_2} \I_j)\otimes\ldots \otimes(\bigotimes_{i_{r-1}<j< i_r} \I_j) \otimes\PP'_{Z_{i_r}} \otimes(\bigotimes_{ i_r<j\leq d} \I_j)-\left({{d}\choose{r}}-1\right)
\bigotimes_{j\leq d} \I_j.$$
In other words, we combine $d$ one-dimensional birth and death chains in such a way, that 
the resulting $d$-dimensional chain can change at most $r$ coordinates by $\pm 1$ at one step.
\par \noindent 
We can rewrite this formula for some cases:
\begin{itemize}
\item  $r=d$, independent games $$ \PP'_{Z}=\bigotimes_{j\leq d} \PP'_{Z_j}.$$
\item $r=d-1$ $$ \PP'_{Z}=\sum_{k=1}^{d} (\bigotimes_{j< k} \PP'_{Z_j}) \otimes \I_k\otimes (\bigotimes_{j>k} \PP'_{Z_j})-(d-1)\bigotimes_{j\geq d} \I_j.$$
 \item $r=2$ $$ \PP'_{Z}=\sum_{k=1}^{d} \sum_{i=k+1}^{d}(\bigotimes_{j< k} \I_j)\otimes \PP'_{Z_k}\otimes (\bigotimes_{k<j<i} \I_j)\otimes\PP'_{Z_i}\otimes (\bigotimes_{i<j \leq d} \I_j)-\left({{d}\choose{2}}-1\right)\bigotimes_{j\geq d} \I_j.$$
 \item $r=1$ $$ \PP'_{Z}=\sum_{k=1}^{d} (\bigotimes_{j< k} \I_j) \otimes\PP'_{Z_k}\otimes (\bigotimes_{j>k} \I_j)-(d-1)\bigotimes_{j\geq d} \I_j.$$
 This  can be rewritten as
$$ \PP'_{Z}=\bigoplus_{j\leq d} \PP'_{Z_j} - (d-1)\bigotimes_{j\leq d} \I_j.$$
$\PP_{Z}=\mathcal{F}_{-\infty}(\PP'_{Z})$ are exactly the transition 
corresponding to the generalized gambler's ruin problem given in (\ref{eq:PZp}). 
\end{itemize}

In all above cases, the winning probability 
for chain   $\PP_{Z}$ is given in  (\ref{eq:gabmler_rho}). 
This is since the winning probabilities for $\PP_{Z_j}$ are given in  (\ref{eq:rho_bd}), thus using 
(\ref{eq:rho_product}) the  relation is (\ref{eq:gabmler_rho}) proven.
\par 

In all above cases, if we replace $\PP_{Z_j}'$ with $\PP_{\hat{X}_j}$ and 
$\PP_{Z}'$ with $\PP_{\hat{X}}$, then we have a special cases for formula for $\PP_{\hat{X}}$ 
given in (\ref{eq:abs_time_res_prod}). 
If, in addition, we assume that $\nu^*((1,\ldots,1))$, then, from Corollary \ref{cor:abc} c) we have 
 $$
T^*_{(1,\ldots,1),\N}=\left\{ 
 \begin{array}{llll}
\hat{T}_{(1,\ldots,1),\N} & \mathrm{with\ probability\ }  \prod_{j=1}^d \rho_j(1), \\[8pt]
+\infty   & \mathrm{with\ probability\ }  1-\prod_{j=1}^d \rho_j(1). \\ 
\end{array}
\right.  $$  
\par 
For example, in case $r=1$ (then we have  $m=d+1$ and take $b_k=1, k=1,\ldots,d, b_{d+1}=1-d, \A_k=\{k\}, k=1,\ldots,d$ and $\A_{d+1}=\emptyset$)
we have 
\begin{equation*}\label{eq:abs_time_res_prod_multidim_gambler}
 \PP_{\hat{X}}((i_1,\ldots,i_d),(i'_1,\ldots, i'_d)) = 
\left\{
\begin{array}{lll}
\displaystyle 1-\lambda_{i_j}^{(j)} & \textrm{if } \ i'_j=i_j+1\\[10pt]
\displaystyle 1-\sum_{j: i_j<N_j} \left(1- \lambda_{i_j}^{(j)}\right) & \textrm{if } \ i'_j=i_j, j=1,\ldots,d\\[10pt]
0 & \textrm{otherwise}.
\end{array}
\right. 
\end{equation*}
Sample transitions for case $d=2, r=1$ are depicted in Fig. \ref{fig:abs_time_d2}.
\medskip\par 
In Figure \ref{fig:abs_time_d3} the transition of $\hat{X}$ are presented for $d=3$:
\begin{itemize}
 \item When $r=1$ only \textcolor{blue}{blue} are possible.
 \item When $r=2$ only \textcolor{blue}{blue} and \textcolor{mygreen}{green} are possible.
 \item When $r=3$ all transitions, \textcolor{blue}{blue}, \textcolor{mygreen}{green} and  \textcolor{red}{red} are possible. 
\end{itemize}

\begin{figure}
\begin{tikzpicture}
\begin{axis}[
    xmin=0,xmax=6.5,
    ymin=0,ymax=6.5,
    grid=both,
    grid style={line width=.1pt, draw=gray!25},
    major grid style={line width=.2pt,draw=gray!25},
    axis lines=middle,
    minor tick num=1,
    enlargelimits={abs=1.0},
    axis line style={latex-latex},
    ticklabel style={font=\tiny,fill=white},
    xticklabels={1,2,3,5,7},
    yticklabels={1,2,3,5,7},
    xlabel style={at={(ticklabel* cs:1)},anchor=north west},
    ylabel style={at={(ticklabel* cs:1)},anchor=south west}
]


  \node at (axis cs: 3,-0.4) {\small{$i_1$}};  
  \node at (axis cs: 7,-0.4) {\small{$N_1$}};  
  \node at (axis cs: -0.6,4) {\small{$i_2$}};
  \node at (axis cs: -0.6,7) {\small{$N_2$}};  
 
 \draw[-,gray!70] (axis cs: 7,0) -- (axis cs: 7,7)  ;
 \draw[-,gray!70] (axis cs: 0,7) -- (axis cs: 7,7)  ;

 
 \draw[->] (axis cs: 3,4) -- (axis cs: 4,4)  node [midway, right=0.4cm] (TextNode1) {\small{$p_1(i_1)$}};
 \draw[->] (axis cs: 3,4) -- (axis cs: 2,4)  node [midway, left=0.4cm] (TextNode2) {\small{$q_1(i_1)$}};
 
 \draw[->] (axis cs: 3,4) -- (axis cs: 3,5)  node [midway, above=0.4cm] (TextNode3) {\small{$p_2(i_2)$}};
 \draw[->] (axis cs: 3,4) -- (axis cs: 3,3)  node [midway, below=0.4cm] (TextNode4) {\small{$q_2(i_2)$}};
 
  \node at (axis cs: 7,7) (A) {$\bullet$};
 
\path (A) edge [anchor=center,loop below] node {1} (A);


\end{axis}
\begin{axis}[xshift=8.5cm,
    xmin=0,xmax=6.5,
    ymin=0,ymax=6.5,
    grid=both,
    grid style={line width=.1pt, draw=gray!25},
    major grid style={line width=.2pt,draw=gray!25},
    axis lines=middle,
    minor tick num=1,
    enlargelimits={abs=1.0},
    axis line style={latex-latex},
    ticklabel style={font=\tiny,fill=white},
       xticklabels={1,2,3,5,7},
    yticklabels={1,2,3,5,7},
    xlabel style={at={(ticklabel* cs:1)},anchor=north west},
    ylabel style={at={(ticklabel* cs:1)},anchor=south west}
]


\draw [gray!50,fill, opacity=0.3] (axis cs:0,0) rectangle (axis cs:3,4);

    \node at (axis cs: 3,-0.4) {\small{$i_1$}};  
  \node at (axis cs: 7,-0.4) {\small{$N_1$}};  
  \node at (axis cs: -0.6,4) {\small{$i_2$}};
  \node at (axis cs: -0.6,7) {\small{$N_2$}};  
 
 \draw[-,gray!70] (axis cs: 7,0) -- (axis cs: 7,7)  ;
 \draw[-,gray!70] (axis cs: 0,7) -- (axis cs: 7,7)  ;
 
 
 \node at (axis cs: 7,7) (A) {$\bullet$};
 
\path (A) edge [anchor=center,loop below] node {1} (A);

 
  
 \draw[->] (axis cs: 3,4) -- (axis cs: 4,4)  node [midway, right=0.5cm] (TextNode1) {{$1-\lambda_{i_1}^{(1)} $}}; 
 \draw[->] (axis cs: 3,4) -- (axis cs: 3,5)  node [midway, above=0.5cm] (TextNode3) {{$1-\lambda_{i_2}^{(2)}$}};


\end{axis}

%
%
      
%
%

\end{tikzpicture}
\caption{Sample transitions for the example from Section  \ref{sec:example_b} with $d=2$: $X^*$ (left) and $\hat{X}$ (right).
Probabilities of staying are not depicted. If $X^*$ starts at $(1,1)$, so does the $\hat{X}$ and $T^*_{(1,1),\N}=\hat{T}_{(1,1),\N} $ provided $q_j(1)=0, j=1,\ldots,d$.
If, say,  $\nu^*((i_1,i_2))=1$ then the pgf of $T^*_{(1,1),\N}$ is a mixture of pgfs of $\hat{T}_{(j_1,j_2),\N}$ for $j_1\leq i_1, j_2\leq i_2$ (shaded area).} \label{fig:abs_time_d2}
\end{figure}
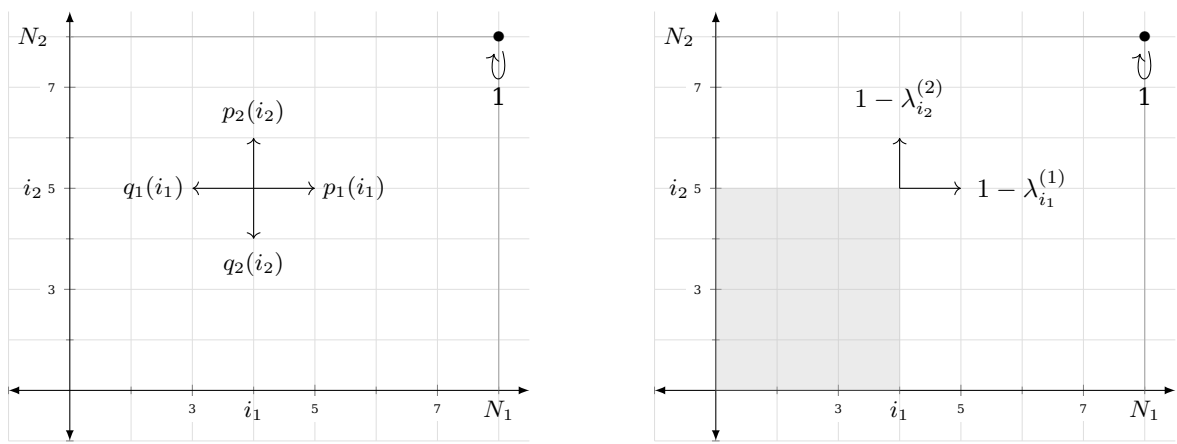



%

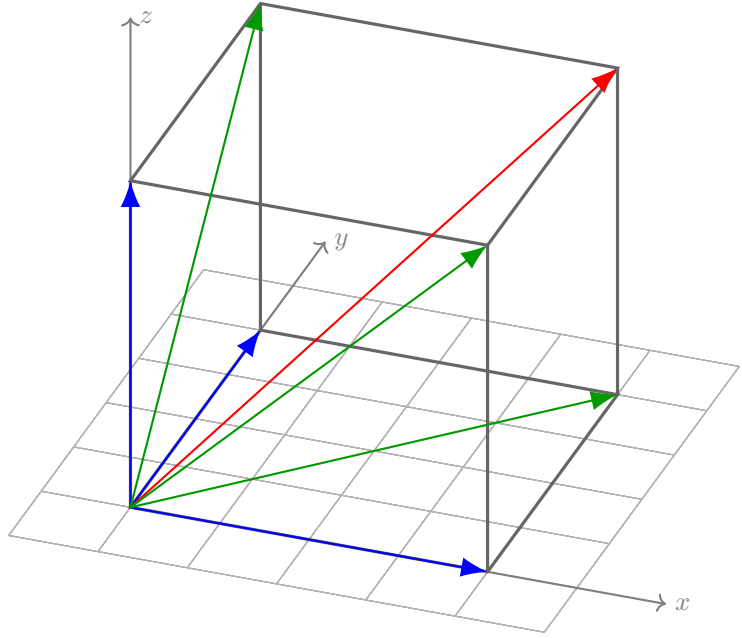
\begin{figure}  
\begin{center}
\tdplotsetmaincoords{60}{20}
\begin{tikzpicture}
		[tdplot_main_coords,
			cube/.style={very thick,black!60},
			grid/.style={very thin,gray!60},
			axis/.style={->,gray,thick},
			walk_blue/.style={->,blue,thick, -{Latex[scale=1.5]}},			
			walk_green/.style={->,black!40!green,thick, -{Latex[scale=1.5]}},			
			walk_red/.style={->,red,thick, -{Latex[scale=1.5]}},			
			scale = 2.5
			]

	\foreach \x in {-0.5,0,...,2.5}
		\foreach \y in {-0.5,0,...,2.5}
		{
			\draw[grid] (\x,-0.5) -- (\x,2.5);
			\draw[grid] (-0.5,\y) -- (2.5,\y);
		}

	\draw[axis] (0,0,0) -- (3,0,0) node[anchor=west]{$x$};
	\draw[axis] (0,0,0) -- (0,3,0) node[anchor=west]{$y$};
	\draw[axis] (0,0,0) -- (0,0,3) node[anchor=west]{$z$};

	\draw[cube] (0,0,0) -- (0,2,0) -- (2,2,0) -- (2,0,0) -- cycle;
	\draw[cube] (0,0,2) -- (0,2,2) -- (2,2,2) -- (2,0,2) -- cycle;
	
	\draw[cube] (0,0,0) -- (0,0,2);	
	\draw[cube] (0,2,0) -- (0,2,2);
	\draw[cube] (2,0,0) -- (2,0,2);
	\draw[cube] (2,2,0) -- (2,2,2);

	\draw[walk_blue] (0,0,0) -- (2,0,0);	
	\draw[walk_blue] (0,0,0) -- (0,2,0);	
	\draw[walk_blue] (0,0,0) -- (0,0,2);	
	
	\draw[walk_red] (0,0,0) -- (2,2,2);

        \draw[walk_green] (0,0,0) -- (0,2,2);	
        \draw[walk_green] (0,0,0) -- (2,0,2);	
        \draw[walk_green] (0,0,0) -- (2,2,0);

	
\end{tikzpicture}
\caption{Sample transitions of $\hat{X}$ for the example from Section  \ref{sec:example_b} with $d=3$: 
transitions for $r=1$ (\textcolor{blue}{blue}), $r=2$ (\textcolor{blue}{blue} and \textcolor{mygreen}{green}) 
and $r=3$ (\textcolor{blue}{blue}, \textcolor{mygreen}{green} and \textcolor{red}{red}) } \label{fig:abs_time_d3}
\end{center}
\end{figure}
\end{example}
  
\subsection{One dimensional gambler's ruin problem related to Ehrenfest model: calculating $T^*_{m,N}$}\label{sec:example_c}

Here we present a concrete example of a birth and death chain on $\E=\{1,\ldots,N\}$ with $N$ being
the only absorbing state, for which we provide $pgf$ of absorption time provided chain started at any $m\in \E$.
We use Lemma \ref{lem:duality_two_linksC} calculating link $\Lambda$. As far as we are aware, this $pgf$ 
cannot be given using results from \cite{Gong2012} \ie (\ref{eq:gTi}). 
This is since the eigenvalues the presented matrix $\PP_{X^*}$ are known, 
but the eigenvalues of $ \PP_{X^*}^{\ceil{m}}$ are not known for any $m\in\E$.
\medskip\par

\begin{example}  
Let $X^*$ be a Markov Chain on the state space $ \E=\{1,\ldots, N\}$ with transition matrix ${\PP}_{X^*}$ of the form:
 $$
 {\PP}_{X^*}(i,i')=\left\{
 \begin{array}{llll}  
  \frac{N-i}{2N-2}\frac{\sum_{r=0}^{i-1} {{N-1} \choose {r-1}} }{\sum_{r=0}^{i} {{N-1} \choose {r-1}}}   & \mathrm{if }\  i'=i-1, i<N, \\[12pt]
  \frac{N-2}{2N-2} & \mathrm{if} \ i'=i, i < N, \\[12pt]
  1 & \mathrm{if} \ i'=i= N, \\[12pt]
   \frac{i}{2N-2}\frac{\sum_{r=0}^{i+1} {{N-1} \choose {r-1}} }{\sum_{r=0}^{i} {{N-1} \choose {r-1}}}   & \mathrm{if }\  i'=i+1, \\[12pt]
  0 & \mathrm{otherwise}, \\
 \end{array}\right.
 $$
Then the absorption time starting at $m\in\E $ is has the following $pgf$:
 \begin{equation}\label{eq:ex_c_pgf}
 pgf_{T_{m,d}^*}(s)=\sum_{j\leq m} \hat{\nu}(j) pgf_{\hat{T}_{j,d}}(s),
\end{equation}
where 
\begin{equation}\label{eq:hat_mu}
  \hat{\nu}(j) = \frac{2^{j-1}(-1)^{m+j}(m-j+1){{N-1}\choose{m}}{{m}\choose{j-1}}}{(N-j)\sum_{k=0}^{m-1}{{N-1}\choose{k}}}
 \quad and \quad 
 pgf_{\hat{T}_{j,N}}(s)=\prod_{k=j}^{N-1} \left[{(1-\frac{k-1}{N-1})s\over 1-\frac{k-1}{N-1} s}\right].
\end{equation}
 In particular, we have 
 \begin{equation}\label{eq:ex_c_ex}
E(T^*_{m,N})=(N-1)\sum_{j\leq m} \hat{\nu}(j) \sum_{k=j}^{N-1}{1\over N-k}.
\end{equation}

\end{example}
\begin{proof}
Let 
  $$
 {\PP}_{\hat{X}}(i,i')=\left\{
 \begin{array}{llll}  
  \frac{i-1}{N-1} & \mathrm{if} \ i'=i, \\[9pt]
  \frac{N-i}{N-1}  & \mathrm{if }\  i'=i+1, \\[9pt]
  0 & \mathrm{otherwise}. \\
 \end{array}\right.
 $$ 
To show the result via Lemma \ref{lem:duality_two_linksC} it is enough 
to find $\Lambda$ such that $\Lambda \PP_{X^*}=\PP_{\hat{X}}\Lambda$ and $\nu^*\Lambda^{-1}=\hat{\nu},$
where $\nu^*(j)=\mathbbm{1}\{j=m\}$.
\par 
However, since $X^*$ has only one absorbing state, we can -- and we will -- follow the \textsl{outline of an alternative proof} 
given in Section \ref{sec:alt}. \textsl{I.e.,} we will indicate intermediate ergodic chain $X$  on $\E$ 
with transition matrix $\PP_X$ and find $\Lambda^*$ and $\hat{\Lambda}$ such that $\Lambda^*\PP_X=\PP_{X^*}\Lambda^*$
and $\hat{\Lambda}\PP_X=\PP_{\hat{X}}\hat{\Lambda}$. Then, we will automatically have 
$\Lambda=\hat{\Lambda} (\Lambda^*)^{-1}$ and we will show that $\nu^*\Lambda^{-1}=\nu^*\Lambda^*\hat{\Lambda}^{-1}=\hat{\nu}$.
\medskip\par 
\noindent
Let $X$ be a chain on $\E$ with the following transition matrix: 
$$
 {\PP}_{X}(i,i')=\left\{
 \begin{array}{llll}  
   \frac{i-1}{2(N-1)} & \mathrm{if }\  i'=i-1, \\[9pt]
  {1\over 2} & \mathrm{if} \ i'=i, \\[9pt]
  \frac{N-i}{2(N-1)}  & \mathrm{if }\  i'=i+1, \\[9pt]
  0 & \mathrm{otherwise}, \\
 \end{array}\right.
 $$
 \ie $X$ corresponds to Ehrenfest model of $N-1$ particles with an extra probability (half) of staying
 (and states are enumerated $1,\ldots,N$, whereas in the classical Ehrenfest model these are $0,\ldots,N-1$).
 Its stationary distribution is a binomial distribution $\pi(j)={1\over 2^{N-1}}{N-1 \choose j-1}$,
  thus the classical link (cf.  (\ref{eq:link_classic})) is given by 
 $$ \Lambda^*(i,j)=\frac{{{N-1}\choose{j-1}}}{\sum_{r=0}^{i-1}{{N-1}\choose{r}}}\mathbbm{1}\{ j \leq i\},$$
 \ie we have  $\Lambda^*\PP_X=\PP_{X^*}\Lambda$ ($X^*$ is a sharp SSD of $X$). 
%
 The eigenvalues of   $X$ are known, these are ${i\over N-1}, i=0,\ldots,N-1$,
 thus $\hat{X}$ is its pure birth spectral dual. The link $\hat{\Lambda}$ such that  $\hat{\Lambda}\PP_X
 =\PP_{\hat{X}}\hat{\Lambda}$ is known
 (see Eq. (4.6) in \cite{Fill2009}), it is given by
 $$\hat{\Lambda}(i,j)={ {i-1\choose j-1}\over 2^{i-1}}.$$
  It can be checked that 
   $$ \hat{\Lambda}^{-1}(i,j)=(-1)^{j-i}2^{j-1}{{i-1}\choose{j-1}}.$$
Note that $i$-th row of $\hat{\Lambda}^{-1}$ corresponds to the 
coefficients\footnote{The on-line encyclopedia of integer sequences. Sequence   \href{http://oeis.org/A303872}{A303872}.} in expansion  of $(2x-1)^{i-1}$.
%
%
 
\noindent
Thus, as outlined in Section \ref{sec:alt} we have $\Lambda \PP_{X^*}=\PP_{\hat{X}}\Lambda$ 
 with $\Lambda=\hat{\Lambda}(\Lambda^*)^{-1}$. 
 We need only to check that $\nu^*\Lambda^{-1}=\nu^*{\Lambda^*}\hat{\Lambda}^{-1}$
 is equal to $\hat{\nu}$ given in (\ref{eq:hat_mu}).
 We have 
 $$\begin{array}{llll}
  \hat{\nu}(j)&=&\displaystyle (\nu^*{\Lambda^*}\hat{\Lambda}^{-1})(j) = \sum_{k} {\Lambda^*}(m,k)\hat{\Lambda}^{-1}(k,j)  \\[20pt]
  &=&\displaystyle \sum_{j\leq k \leq m}\frac{{{N-1}\choose{k-1}}}{\sum_{r=0}^{m-1}{{N-1}\choose{r}}} (-1)^{j-k}2^{j-1}{{k-1}\choose{j-1}}\\[12pt]
  &=&\displaystyle {2^{j-1} \over {\sum_{r=0}^{m-1}{{N-1}\choose{r}}} }\sum_{j\leq k \leq m}(-1)^{j-k}{{N-1}\choose{k-1}} {{k-1}\choose{j-1}}\\[20pt]
   &\stackrel{(*)}{=}&\displaystyle{2^{j-1} {{N-1}\choose{j-1}} \over {\sum_{r=0}^{m-1}{{N-1}\choose{r}}} }\sum_{j\leq k \leq m}(-1)^{j-k} {{N-j}\choose{N-k}}   
 \end{array}
$$
 where in $\stackrel{(*)}{=}$ we used identity ${N-1\choose k-1}{k-1\choose j-1}={N-1\choose j-1}{N-j\choose N-k}$.
 As for the last sum we have 
$$
\begin{array}{lll}
  \displaystyle \sum_{j\leq k \leq m}(-1)^{j-k} {{N-j}\choose{N-k}} & =& \sum_{j\leq k \leq m}(-1)^{j-k} {{N-j}\choose{k-j}} \\[20pt]
 =\displaystyle  \sum_{k=0}^{m-j}(-1)^{k} {{N-j}\choose{k}}&\stackrel{(**)}{=}& (-1)^{m-j} {N-j-1\choose m-j},
\end{array}
$$
 where in $\stackrel{(**)}{=}$ we used identity\footnote{See \textsl{Partial sums} at 
 \href{https://en.wikipedia.org/wiki/Binomial_coefficient}{https://en.wikipedia.org/wiki/Binomial\_coefficient}}
  $\sum_{k=0}^M(-1)^k{n\choose k}=(-1)^M{n-1\choose M}$.
 Finally,
 $$\hat{\nu}(j)={2^{j-1} (-1)^{m-j} {{N-1}\choose{j-1}} {N-j-1\choose m-j} \over {\sum_{r=0}^{m-j}{{N-1}\choose{r}}} }, 
 $$  
 what is equal to (\ref{eq:hat_mu}).

%
%
%
%
%
%
%
 Note that $pgf$ given in (\ref{eq:hat_mu}) corresponds to the distribution of $\sum_{k=j}^{N-1} Y_k$, where $Y_k$ is a geometric random variable with parameter ${k-1\over N-1}$ and $Y_1,\ldots,Y_{N-1}$ are independent.
 We have $EY_k={N-1\over N-k}$ thus (\ref{eq:ex_c_ex}) follows from (\ref{eq:ex_c_pgf}) and (\ref{eq:hat_mu}).
 
 \end{proof}
 
 \begin{remark}\rm
 Calculating $\hat{\nu}$ we have actually calculated the link $\Lambda$, which is given by
     $$ \Lambda(i,j)=\left\{
 \begin{array}{llll}  
  \frac{2^{j-1}(-1)^{i+j}(i-j+1){{N-1}\choose{i}}{{i}\choose{j-1}}}{(N-j)\sum_{k=0}^{i-1}{{N-1}\choose{k}}} & \mathrm{if} \ j<N, \\
  0 & j=N, i<N, \\[6pt]
  1 & j=N, i =N. \\
 \end{array}\right.
 $$ 
 
 \end{remark}


\ACKNO{
Authors thank Bartłomiej Błaszczyszyn for helpful discussions, in particular for suggesting 
exploiting Kronecker products.}

\end{document}